\documentclass[a4paper,11pt,reqno]{article}

\usepackage{relsize}
\usepackage{cite}
\usepackage{hyperref}
\hypersetup{
  colorlinks   = true, %Colours links instead of ugly boxes
  urlcolor     = blue, %Colour for external hyperlinks
  linkcolor    = Purple, %Colour of internal links%
  citecolor   = red %Colour of citations
}
\usepackage{amsmath,amsthm,amssymb}

\usepackage[margin=2.6cm]{geometry}

\usepackage{color}

\usepackage{stmaryrd}

\theoremstyle{definition}
\newtheorem{theorem}{Theorem}[section]
\newtheorem{definition}[theorem]{{{Definition}}}
\newtheorem{example}[theorem]{{{Example}}}
\newtheorem{notation}[theorem]{{{Notation}}}
\newtheorem{remark}[theorem]{{{Remark}}}

\newtheorem{corollary}[theorem]{{{Corollary}}}%[theorem]
\newtheorem{proposition}[theorem]{{{Proposition}}}
\newtheorem{lemma}[theorem]{{{Lemma}}}

\theoremstyle{definition}
\newtheorem*{lem}{Lemma}

\newcommand{\numberset}{\mathbb}

\newcommand{\F}{\numberset{F}}

\newcommand{\C}{\mathcal{C}}

\newcommand{\mC}{\mathcal{C}}
\newcommand{\mP}{\mathcal{P}}

\newcommand{\mB}{\mathcal{B}}

\newcommand{\mM}{\mathcal{M}}

\newcommand{\mU}{\mathcal{U}}
\newcommand{\mQ}{\mathcal{Q}}

\newcommand{\dH}{d^\textnormal{H}}
\newcommand{\sH}{\sigma^\textnormal{H}}
\newcommand{\mE}{\mathcal{E}}

\newcommand{\mV}{\mathcal{V}}

\newcommand{\wt}{\textnormal{wt}}
\newcommand{\rs}{\textnormal{rowsp}}

\newcommand{\Fq}{\F_q}
\newcommand{\Fm}{\F_{q^m}}

\newcommand{\EE}{{\numberset{E}}}
\newcommand{\Var}{{\textnormal{Var}}}

\newcommand\qbin[3]{\binom{#1}{#2}_{#3}}

\newcommand{\HH}{\textnormal{H}}
\newcommand{\rk}{\textnormal{rk}}

\newcommand{\drk}{d^\textnormal{rk}}

\newcommand{\maxrk}{w^\textnormal{rk}}

\DeclareMathOperator{\GL}{GL}

\DeclareMathOperator{\PG}{PG}
\DeclareMathOperator{\PGL}{PGL}
\DeclareMathOperator{\mm}{m}
\DeclareMathOperator{\Ext}{Ext}

\newcommand{\wH}{\omega^\textnormal{H}}

\newcommand{\Fmkd}{[n,k,d]_{q^m/q}}
\newcommand{\Fmk}{[n,k]_{q^m/q}}
\newcommand{\st}{\, : \,}
\usepackage[dvipsnames]{xcolor}

\definecolor{lgray}{gray}{0.85}

\numberwithin{equation}{section}
\allowdisplaybreaks

\usepackage{titling}
\renewenvironment{thebibliography}[1]{
  \begin{oldthebibliography}{#1}
    \setlength{\itemsep}{0.17cm}
    \setlength{\parskip}{0em}
}
{
  \end{oldthebibliography}
}

\title{\textbf{Linear Cutting Blocking Sets \\ and Minimal Codes in the Rank Metric}}

\usepackage{authblk}
\author[1]{Gianira N. Alfarano\thanks{Gianira N. Alfarano  is supported by the Swiss National Science Foundation through grant no. 188430.}}
\affil[1]{Institute of Mathematics, University of Zurich, Switzerland}

\author[2]{Martino Borello}
\affil[2]{Universit\'e Paris 8, Laboratoire de G\'eom\'etrie, Analyse et Applications, LAGA,
Universit\'e Sorbonne Paris Nord, CNRS, UMR 7539, France}

\author[3]{Alessandro Neri}%\thanks{Alessandro Neri was supported by the Swiss National Science Foundation through grant no. 187711.}}
\affil[3]{Max-Planck-Institute for Mathematics in the Sciences, Leipzig, Germany}

\author[4]{Alberto Ravagnani}
\affil[4]{Department of Mathematics and Computer Science, Eindhoven University of Technology, the Netherlands}

\date{}

\usepackage{setspace}
\usepackage{enumitem}
\setitemize{itemsep=0em}
\setenumerate{itemsep=0em}

\begin{document}

\maketitle

\begin{abstract}
This work investigates the structure of rank-metric codes in connection with concepts from finite geometry, most notably 
the $q$-analogues of projective systems and blocking sets.
We also illustrate how to associate a classical Hamming-metric code to a rank-metric one, in such a way that various rank-metric properties naturally translate into the homonymous Hamming-metric notions under this correspondence.
The most interesting applications of our results lie in the theory of minimal rank-metric codes, 
which we introduce and study from several angles. Our main contributions are bounds for the parameters of a minimal rank-metric codes, a general existence result based on a combinatorial argument, and an explicit code construction for some parameter sets that uses the notion of a scattered linear set.
Throughout the paper we also show and comment on curious analogies/divergences between the theories of error-correcting codes in the rank and in the Hamming metric.
\end{abstract}

\bigskip

\section{Introduction}
\label{sec:1}

Block codes with the Hamming metric have been extensively (and traditionally) studied in connections with several topics in finite geometry, including arcs, blocking sets,  and modular curves to mention a few; see \cite{MR939470,MR818812,MR1345289} among many others. In the last decade, especially thanks to the advent of network coding \cite{MR1966785,MR1768542,MR2450762,MR2451015}, the novel class of rank-metric codes has been the subject of intense mathematical research. Interesting progress has been recently made in the attempt~of understanding the connection between rank-metric codes and 
finite geometry~\cite{randrianarisoa2020geometric}, yet this link is still not fully understood and rather unexplored.
This paper contributes to fill in this important gap. 

The starting point of our investigation is a connection between rank-metric codes and the $q$-analogues of projective systems. This link has been observed already in~\cite{randrianarisoa2020geometric}, whose contributions we survey with short proofs and extend. Among the various new results, 
we show that the maximum rank of a nondegenerate rank-metric code $\mC \subseteq \F_{q^m}^n$ is $\min\{m,n\}$, a quite simple property that nonetheless has interesting consequences in the theory of anticodes and minimal rank-metric codes (see below).

We then apply the theory of $q$-systems to show how one can associate a Hamming-metric code to a given rank-metric code. 
This correspondence translates various properties of a rank-metric codes into the homonymous properties in the Hamming metric. In particular, the Hamming-metric code associated to the simplex rank-metric code is (essentially) the classical simplex code.

The interplay between the rank and the Hamming metric also motivates us to investigate one of the best-known parameters of a code, namely, its \textit{total weight}. We identify a suitable rank-metric analogue of the total Hamming weight of a code and show that it has a constant value for all nondegenerate rank-metric codes with the same dimension and length.
We then compute its asymptotic behaviour as the field size $q$ tends to infinity, as well as the asymptotic behaviour of its variance under certain assumptions.
This illustrates the general behaviour of these parameters over large finite fields.

Several applications of the above-mentioned results and concepts can be seen in theory of minimal rank-metric codes, a research line which is seemingly unexplored. We call a rank-metric code \textit{minimal} if all its codewords have minimal \textit{rank support}. Minimal rank-metric codes are the natural analogues (in the rank-metric) of minimal Hamming-metric codes, a class of objects that have been extensively studied in connection with finite geometry; see e.g.~\cite{bonini2020minimal,alfarano2019geometric,tang}.

The stepping stone in our approach is a characterization of minimal rank-metric codes via $q$-systems.
The correspondence described above between rank-metric codes and these geometric/combinatorial structures 
induces a correspondence between minimal rank-metric codes and \textit{linear cutting blocking sets}. The latter concept can be regarded as the $q$-analogue of the classical notion of a \textit{cutting blocking set}.

The description of minimal rank-metric codes via the $q$-analogues of cutting blocking sets allows us to
establish a lower bound for their length. More precisely, we find that a minimal rank-metric code $\mC \subseteq \F_{q^m}^n$ of dimension $k$ must satisfy 
\begin{equation}\label{aaa1}
  n \ge k+m-1.  
\end{equation}
We also show that a nondegenerate rank-metric code is minimal if and only if the associated Hamming-metric code is minimal (under the correspondence described earlier). This result naturally connects the theories of minimal codes in the two metrics and makes it possible to transfer/compare results across them. 

A major, rather curious difference between minimal codes in the rank and in the Hamming metric appears to be in the role played by the field size $q$ with respect to bounds and existence results. While in the Hamming metric the field size $q$ is a crucial parameter (e.g., minimal codes do not exist for lengths that are too small compared to a suitable multiple of the field size), most of the bounds and existence results we derive for minimal rank-metric codes do not depend on $q$, even when this quantity explicitly shows up in the computations. We will elaborate more on this later in the paper.

Our main contributions to the theory of minimal codes in the rank metric lies in existence results and constructions, which we now describe very briefly. 
We start by giving simple examples of minimal rank-metric codes (the simplex rank-metric code and nondegenerate codes of very large length). Next, we propose a general construction of 3-dimensional minimal rank-metric codes based on the theory of scattered linear sets. The construction also proves that our lower bound for the length of a minimal rank-metric code is sharp for some (infinite) parameter sets.
We then establish a general existence result for minimal rank-metric codes based on a combinatorial argument. More precisely, we show that a minimal rank-metric code $\mC \subseteq \F_{q^m}^n$ of dimension $k \ge 2$ exists whenever $m \ge 2$ and
\begin{equation} \label{aaa2}
  n \ge 2k+m-2.  
\end{equation}
Comparing~\eqref{aaa1} with \eqref{aaa2} we see that, in general, the existence of minimal rank-metric codes remains an open question only for  $k-1$ values of $n$ (for any fixed $m$, $k$ and $q$).

We conclude the paper by introducing a structural quantity of a $q$-system, which we call its \textit{linearity index}. 
This allows us to attach a new parameter to a rank-metric code $\mC$ via its associated $q$-system. We investigate 
which structural properties of a rank-metric code are captured by its linearity index, showing also a connection with the theory of generalized rank weights.
Finally, we apply these concept to describe further minimal codes in the rank metric. 

\paragraph{Outline.}
The remainder of the paper is organized as follows. Section~\ref{sec:2} contains the preliminaries on rank-metric code, Hamming-metric codes and (cutting) blocking sets. 
In Section~\ref{sec:3} we describe the geometric structure of rank-metric codes via the theory of $q$-systems. In particular, we study simplex rank-metric codes. Section~\ref{sec:4} illustrates how to associate a Hamming-metric code to a rank-metric code and how code properties behave under this correspondence. Sections~\ref{sec:5} and~\ref{sec:6} are entirely devoted to the study of minimal codes in the rank metric: geometric structure, properties, constructions and existence. Finally, the Appendix contains some technical proofs.

\bigskip

%%%%%%%%%%%%%%%%%%%%%%%

\section{Preliminaries}
\label{sec:2}

\subsection{Rank-Metric Codes}
Throughout this paper, $q$ denotes a prime power and $n,m$ are positive integers.
We start by introducing the main object studied in this work, namely, rank-metric codes.

For a vector 
$v \in \F_{q^m}^n$ and an ordered basis $\Gamma=\{\gamma_1,\ldots,\gamma_m\}$ of the field extension
$\F_{q^m}/\F_q$, let $\Gamma(v) \in \F_q^{n \times m}$ be the matrix defined by
$$v_i= \sum_{j=1}^m \Gamma(v)_{ij} \gamma_j.$$
Note that $\Gamma(v)$ is  constructed by simply transposing $v$ and then expanding each entry over the basis $\Gamma$.
The \textbf{$\Gamma$-support} of a vector $v \in \F_{q^m}^n$ is the column space of $\Gamma(v)$.
It is denoted by $\sigma_\Gamma(v) \subseteq \F_q^n$.
The  following result can be obtained by a standard linear algebra argument.
% \ALB{Per un giornale come JCTA, lascerei la prova di questo al lettore.} \GIANI{Va bene toglierla se siete tutti d'accordo. Però secondo me non è completamente banale la prova.}

\begin{proposition} \label{prop:prel}
Let $v \in \F_{q^m}^n$.
\begin{enumerate}
\item We have $\sigma_\Gamma(v)=\sigma_\Gamma(\alpha v)$ for all nonzero
$\alpha \in \F_{q^m}$ and all bases~$\Gamma$.
\item The $\Gamma$-support of $v$ does not depend on the choice of the basis $\Gamma$.
\item For all matrices $A \in \F_q^{n \times n}$ we have $\Gamma(vA)=A^\top \Gamma(v)$.
\end{enumerate}
\end{proposition}

\begin{definition}
In the sequel, for $v \in \F_{q^m}^n$ we let 
$\sigma^\rk(v):=\sigma_\Gamma(v)$
be the (\textbf{rank}) \textbf{support} of $v$,
where $\Gamma$ is \textit{any} basis of $\F_{q^m}/\F_q$.
The support is well-defined by Proposition~\ref{prop:prel}. The \textbf{rank} (\textbf{weight}) of a vector $v$ is the $\F_q$-dimension of its support, denoted by $\rk(v)$. 
\end{definition}

Rank-metric codes and their fundamental parameters are defined as follows. In this paper, we follow~\cite{gabidulin1985theory} and only concentrate on rank-metric codes that are linear over $\F_{q^m}$.

\begin{definition}
A (\textbf{rank-metric}) \textbf{code} is an $\F_{q^m}$-linear subspace $\mC \subseteq \F_{q^m}^n$. Its elements are called \textbf{codewords}.
The integer $n$ is the \textbf{length} of the code. The \textbf{dimension} of $\mC$ is the dimension as an $\Fm$-vector space and the \textbf{minimum} (\textbf{rank}) \textbf{distance} of a nonzero code $\mC$ is
$$\drk(\mC):=\min \{\rk(v) \st v \in \mC, \, v \neq 0\}.$$ We also define the minimum distance of the zero code to be $n+1$.
We say that $\mC$ is an $\Fmkd$ code if it has length $n$, dimension $k$ and minimum distance $d$. When the minimum distance is not known or is irrelevant, we write $\Fmk$. A \textbf{generator matrix} of an $\Fmk$ code is a matrix $\smash{G \in \F_{q^m}^{k \times n}}$ whose rows generate $\mC$ as an $\F_{q^m}$-linear space.
Finally, the (\textbf{rank}) \textbf{support} of an $\Fm$-linear rank-metric code $\mC$ is the sum of the supports of its codewords, i.e., 
 $$\sigma^\rk(\mC)= \sum_{v \in \mC}  \sigma^\rk(v).$$
\end{definition}

The support of a rank-metric codes is determined by the supports of any set of generators, as the following simple result shows.

\begin{proposition}\label{prop:supp_propr}
For every $v,w\in\F_{q^m}^n$, we have $\sigma^\rk(v+w) \subseteq \sigma^\rk(v)+\sigma^\rk(w)$. Moreover, if $\mC = \langle c_1,\dots c_t\rangle_{\F_{q^m}}\subseteq \F_{q^m}^n$ is a rank-metric code, then $\sigma^\rk(\mC) = \sigma^\rk(c_1) + \cdots + \sigma^\rk(c_t).$ \qedhere

\end{proposition}

Recall that a (\textbf{linear}, \textbf{rank-metric}) \textbf{isometry} of~$\Fm^n$ is an $\Fm$-linear automorphism~$\varphi$ of~$\Fm^n$ that preserves the rank weight, i.e., such that $\rk(v)=\rk(\varphi(v))$ for all $v\in\Fm^n$. It is known that the isometry group of~$\Fm^n$, say $\mathcal{G}(q,m,n)$, is generated by the (nonzero) scalar multiplications of $\Fm$ and the linear group $\GL_n(q)$; see e.g. \cite{berger2003isometries}. More precisely, $\mathcal G(q,m,n) \cong \Fm^* \times \GL_n(q)$, which (right-)acts on $\Fm^n$ via
$$\begin{array}{rcl}
    (\Fm^* \times \GL_n(q)) \times \Fm^n & \longrightarrow & \Fm^n \\
    ((\alpha , A), v) &  \longmapsto &  \alpha v A.
\end{array}$$

\begin{definition}\label{del:equiv}
Rank-metric codes $\mC,\mC^\prime \subseteq \F_{q^m}^n$ are (\textbf{linearly}) \textbf{equivalent} if there exists $\varphi\in \mathcal{G}(q,m,n)$ such that $\mC^\prime=\varphi(\mC).$
\end{definition}

Observe that, by $\Fm$-linearity, when studying linear equivalence of $\Fmk$ codes
the action of $\Fm^*$ is trivial.
In particular, $\Fmk$ codes $\mC$ and $\mC^\prime$ are equivalent if and only if there exists $A \in \GL_n(q)$ such that $$\mC^\prime=\mC \cdot A:=\left\{ vA \st v\in \mC\right\}.$$

We conclude this section with the definition of dual code, which we will use often throughout the paper.

\begin{definition}
 The \textbf{dual} of a rank-metric code $\mC \subseteq \F_{q^m}^n$ is the rank-metric code
 $$\mC^\perp=\{v \in \F_{q^m}^n \st u \cdot v^\top  =0 \mbox{ for all $u \in \mC$}\} \subseteq \F_{q^m}^n.$$
 \end{definition}
Recall moreover that $\dim_{\F_{q^m}}(\mC)+\dim_{\F_{q^m}}(\mC^\perp)=n$ for all rank-metric codes $\mC \subseteq \F_{q^m}^n$.
There are several relations between a code and its dual, the most elegant of which are probably the MacWilliams(-type) identities. These were established by Delsarte in~\cite{delsarte1978bilinear} for $\F_q$-linear rank-metric codes endowed with the trace product. A simpler proof and their connection with the theory of $\F_{q^m}$-linear rank-metric codes considered here can be found in~\cite{ravagnani2016rank}. 

\subsection{Hamming-Metric Codes}

In this paper, we will often consider codes endowed with the Hamming metric
%~\cite{lint1992introduction, huffman2010fundamentals}
and compare their behaviour with that of rank-metric codes with respect to several properties. We therefore briefly recall some classical notions from coding theory. For more details the reader is referred to \cite{huffman2010fundamentals,tsfasman1991algebraic}.

\begin{definition}
 The (\textbf{Hamming}) \textbf{support} of a vector 
$v \in \F_q^n$ is  $\sH(v)=\{i \st v_i \neq 0\} \subseteq \{1,\dots,n\}$ and its \textbf{Hamming weight} is $\wH(v)=|\sH(v)|$.

An $[n,k]_q$ \textbf{Hamming-metric code} $\mC$ is an $\F_q$-linear subspace $\mC \subseteq \F_q^n$ of dimension $k$. The \textbf{minimum distance} of
$\mC$ is the integer $\dH(\mC)=\min\{\wH(c) \st c \in \mC, \, c \neq 0\}$. If $d=\dH(\mC)$ is known, we say that $\mC$ is 
an $[n,k,d]_q$ code. A \textbf{generator matrix} of $\mC$ is a matrix  $G\in\F_q^{k\times n}$  whose rows generate $\mC$ as an $\F_q$-linear space. 
Finally, we say that $\mathcal{C}$ and $\mathcal{C}^\prime$ are (\textbf{monomially})  \textbf{equivalent} if there exists an $\F_q$-linear isometry $f: \F_q^n \to \F_q^n$ with $f(\mC)=\mC'$.
%; see \cite[pag. 24]{huffman2010fundamentals}
\end{definition}

%We denote by $\dH(\mC)$ the minimum Hamming distance of a code $\mC$ and by $\wH(v)$ the Hamming weight of a vector.

Recall that the \textbf{Hamming support} $\sH(\mC)$ of a code 
$\mC$ is the union of the supports of its codewords.
The code $\mC$ is called \textbf{Hamming-nondegenerate}
if $\sH(\mC)=\{1,\dots,n\}$ and \textbf{Hamming-degenerate} otherwise.
%.  We say that~$\mC$ is \textbf{Hamming-degenerate} if it is not nondegenerate.

There is a well-known geometric interpretation of codes endowed with the Hamming metric. To describe it, we recall the following setting. The projective geometry $\PG(k-1,q)$ with underlying vector space $\F_q^k$ is defined as 
$$ \PG(k-1,q):= \left(\F_q^{k}\setminus \{0\}\right)/_\sim, $$
where $\sim$ denotes the proportionality relation, i.e.,
$u\sim v$ if and only if $u=\lambda v$
for some nonzero element~$\lambda \in \F_q$. 

\begin{definition}\label{def:projsystem}
 A \textbf{projective} $[n,k,d]_q$ \textbf{system} $(\mathcal{P},\mm)$ is a finite multiset, where $\mP\subseteq \PG(k-1,q)$ is a set of points that do not all lie on a hyperplane, and $\mm:\PG(k-1,q)\rightarrow \mathbb N$ is the multiplicity function, with $\mm(P)>0$ if and only if $P\in\mP$ and $\sum_{P\in\mathcal P}\mm(P)=n$.
 The parameter~$d$ is defined as 
 $$d = n- \max\bigg\{\sum_{P\in H}\mm(P) \st  H\subseteq \PG(k-1,q), \; \dim(H) = k-2\bigg\}.$$ 
 Projective $[n,k,d]_q$ systems $(\mathcal{P},\mm)$ and $(\mathcal{P}^\prime,\mm^\prime)$ are \textbf{equivalent} if there exists a projective isomorphism~$\phi \in \PGL(k,q)$ mapping $\mathcal{P}$ to $\mathcal{P}^\prime$ that preserves the multiplicities of the points, i.e., such that $\mm(P)=\mm^\prime(\phi(P))$ for every $P\in\PG(k-1,q)$.
\end{definition}

There exists a 1-to-1 correspondence between (monomial) equivalence classes of $[n,k,d]_{q}$ Hamming-nondegenerate codes and equivalence classes of projective $[n,k,d]_q$  systems; see e.g. \cite[Theorem 1.1.6]{tsfasman1991algebraic}. The correspondence can be formalized by two maps
$$ \begin{array}{rcl} \Phi^\HH: C[n,k,d]_q & \longrightarrow & \mP[n,k,d]_{q}, \\
%  \left[(u_1 \mid \cdots \mid u_n)\right] & \longmapsto  & \{\{ [u_1],\ldots,[u_n]\}\} \\
\Psi^\HH: \mP[n,k,d]_{q} & \longrightarrow & C[n,k,d]_{q},
\end{array}$$
which are the inverse of each other. For a given equivalence class $[C]$ of nondegenerate $[n,k,d]_{q}$ codes, choose a generator matrix $G\in \F_q^{k \times n}$ of one of the codes in $[C]$. Let $g_1,\ldots,g_n$ be the columns of $G$ and take the set $\mP=\{[g_1],\ldots,[g_n]\}\subseteq \PG(k-1,q)$. Moreover, define the multiplicity function $\mm$ as
$$\mm(P)= |\{ i \st P=[g_i]\}|.$$
Then, the map $\Phi^\HH$ is defined to be $\Phi^\HH([C])=[(\mP,\mm)]$.
On the other hand, for a given equivalence class $[(\mP,\mm)]$ of projective $[n,k,d]_{q}$ systems, we construct a matrix $G$ by taking as columns representatives of the points $P_i$'s in $\mP$, each counted with multiplicity $\mm(P_i)$. We then set $\Psi^\HH([\mP,\mm])=[\rs(G)]$.

\begin{definition}
 Let $\mC$ be an $[n,k]_q$ code. A codeword $c\in\mC$ is \textbf{Hamming-minimal} if every nonzero codeword $c^\prime$ with $\sH(c^\prime)\subseteq\sH(c)$ is a multiple of $c$. A code is \textbf{Hamming-minimal} if all its codewords are Hamming-minimal. 
\end{definition}

Minimal codes in the Hamming metric have been extensively studied not only for their applications to secret sharing schemes \cite{Massey}, but also for their geometric and combinatorial properties. 
We recall the following definition.

\begin{definition}
A $t$-\textbf{fold} \textbf{blocking set} in~$\PG(k-1,q)$ is a set $\mP\subseteq \PG(k-1,q)$ such that for every hyperplane $H$ of $\PG(k-1,q)$ we have $|H \cap \mP|\geq t$. When $t=1$ we call $\mP$ a \textbf{blocking set}. A blocking set $\mathcal{P}$ is called \textbf{cutting} if for every hyperplane $H$ of $\PG(k-1,q)$ we have $\langle \mathcal{P} \cap H\rangle = H$. 
\end{definition}

Cutting blocking sets have been introduced in connection to minimal codes in \cite{bonini2020minimal}. However, the same objects were already known under different names and analyzed under different point of view. In \cite{1930-5346_2011_1_119}, they were called  \emph{strong blocking sets} and used for constructing saturating sets in projective spaces over finite fields. Moreover, they were also known as \emph{generator sets} and constructed as union of disjoint lines in \cite{fancsali2014lines}.

The following result relates cutting blocking sets and Hamming-minimal codes and can be found in \cite{alfarano2019geometric,tang}. 

\begin{theorem}\label{thm:cuttingHammingminimal} 
The maps $\Psi^\HH$ and $\Phi^\HH$ define a 1-to-1  correspondence between equivalence classes of cutting blocking sets and equivalence classes of nondegenerate Hamming-minimal codes. 
\end{theorem}

%%%%%%%%%%%%%%%%%%%%%%%

\section{The Geometry of Rank-Metric Codes}
\label{sec:3}

In this section we study the geometric structure of rank-metric codes and their connection with the theory of $q$-systems, introducing fundamental tools that will be needed later. We also describe one-weight and simplex codes in the rank metric.

Although most of the results contained in this section have already appeared in \cite{sheekey2019scatterd} and \cite{randrianarisoa2020geometric}, we are not aware of any organic survey of the topic, which we offer here. For convenience of the reader, we include concise proofs and state results in the form that will be needed in later sections of the paper.

\subsection{Geometric Characterization of Rank-Metric Codes}\label{sec:geo-rk}
We start by introducing the natural analogue of the notion of ``nondegenerate'' code in the rank-metric setting.

\begin{definition}\label{def:nondegenerate}
 An $\Fmk$ rank-metric code $\mC$ is (\textbf{rank-})\textbf{nondegenerate} if $\sigma^\rk(\mC)=\F_q^n$.
 We say that $\mC$ is (\textbf{rank-})\textbf{degenerate} if it is not nondegenerate. Moreover, we call $\dim(\sigma^\rk(\mC))$ 
 the \textbf{effective length} of the code $\mC$.
\end{definition}

\begin{proposition}\label{prop:newnondegeneracy}
Let $\mC\subseteq\Fm^n$ be a rank-metric code. The following are equivalent.
\begin{enumerate}
    \item $\mC$ is rank-nondegenerate.
    \item For every $A \in \GL_n(q)$, the code $\mC \cdot A$ is Hamming-nondegenerate.
    \item The $\F_q$-span of the columns of any generator matrix of $G$ has $\Fq$-dimension $n$.
    \item $\drk(\mC^\perp) \ge 2$.
\end{enumerate}
\end{proposition}

 \begin{proof}
\underline{$(1) \Rightarrow (2)$}: Assume that $\mC \cdot A$ is Hamming-degenerate for some $A \in \GL_n(q)$. Then there exists $1 \le i \le n$ with
$(vA)_i=0$ for all $v \in \mC$. In particular, $\sigma^\rk(vA) \subseteq V:= \langle e_j \st j \neq i\rangle$.
Using Proposition~\ref{prop:prel}, we see that $\sigma^\rk(\mC)$ is contained in an $(n-1)$-dimensional subspace of~$\F_q^n$, hence~$\C$ is rank-degenerate.

\underline{$(2) \Rightarrow (4)$}:  Let $\Gamma:=\{\gamma_1,\dots, \gamma_m\}$ be an $\Fm/\Fq$ basis. If $d(\mC^\perp)=1$, then there exists $v\in\mC^\perp$ with $\rk(\Gamma(v))=1$. Therefore there exists 
$A\in\GL_n(q)$ with $v=(0,\dots, 0,1)\in(\mC\cdot A)^\perp$.
Thus~$\mC\cdot A$ is Hamming-degenerate code.

\underline{$(4) \Rightarrow (1)$}: A
rank-degenerate code $\mC$ is equivalent to a code~$\mC \cdot A$ in which 
all codewords have a~$0$ in the last component. Hence $(0,\dots, 0,1)\in (\mC\cdot A)^\perp$ and $d(\mC^\perp)=1$.

\underline{$(2) \Rightarrow (3)$}: Let $G$ be a generator matrix of $\C$. Since 
$\mC\cdot A$ is Hamming-nondegenerate for
any $A\in\GL_n(q)$, the columns of $G$ are linearly independent over $\Fq$. This implies that $n=\dim(\sigma^\rk(\mC))$ is equal to the dimension of the $\Fq$-space of the columns of $G$.

\underline{$(3) \Rightarrow (1)$}: This immediately follows from the definition of rank-nondegenerate code. \qedhere
\end{proof}

\begin{remark}\label{rem:effectivelength}
By Proposition \ref{prop:newnondegeneracy}, a degenerate code can be isometrically embedded in~$\Fm^{n^\prime}$, where $n^\prime =\dim(\sigma^\rk(\mC))$.
\end{remark}

The following result shows that the parameters of a nondegenerate code must obey certain constraints.

\begin{proposition}(see \cite[Corollary 6.5]{jurrius2017defining})\label{prop:nleqkm}
 Let $\mC$ be an $\Fmk$ nondegenerate rank-metric code. Then $n\leq km$.
\end{proposition}
\begin{proof}
Let $\{c_1,\dots,c_k\}$ be a set of generators for $\mC$. Then, by Proposition \ref{prop:supp_propr}, $\sigma^\rk(\mC)$ is generated by $\sigma^\rk(c_i)$ for $i=1,\dots,k$. Since $\dim(\sigma^\rk(c_i)) \leq m$ for all $i$ and $\sigma^\rk(\mC)=\Fq^n$,
we conclude that $n\leq km$.
\end{proof}

Our next move is to identify geometric objects able to capture the structure of rank-metric codes. We re-formulate the definition of  $q$-analogue of a projective system proposed in \cite{randrianarisoa2020geometric} as follows.

\begin{definition}
 An \textbf{$[n,k,d]_{q^m/q}$  system} is an $n$-dimensional $\Fq$-space $\mU\subseteq \Fm^k$ with the properties that $\langle \mU\rangle_{\Fm}=\Fm^k$ and 
 \begin{equation} \label{d=}
 d = n- \max \left\{\dim_{\Fq}(\mU\cap H) \st H \mbox{ is an } \Fm\mbox{-hyperplane of }  \Fm^k \right\}.
 \end{equation}
 Note that \eqref{d=} can be re-written as 
\begin{equation*}
    \min \left\{\dim_{\Fq}(\mU + H) \st H \mbox{ is an } \Fm\mbox{-hyperplane in }  \Fm^k  \right\} -m(k-1).
\end{equation*}
When the parameters are not relevant, we simply call such an object a \textbf{$q$-system}.
\end{definition}

Two $\Fmk$ systems $\mU, \mV$ are said to be \textbf{equivalent} if there exists an
$\F_{q^m}$-isomorphism $\phi: \F_{q^m}^k \to \F_{q^m}^k$  such that $\phi(\mU)=\mV$. 

\bigskip
The following simple result is a geometric formulation of one of the \emph{Standard Equations}
(stated in our context), which will be of great help throughout the paper. Recall that for integers $a \ge b \ge 0$ and a prime power $Q$, the symbol
$$\binom{a}{b}_{Q}$$
denotes the number of $b$-dimensional subspaces of an $a$-dimensional space over $\F_Q$. This quantity is called a \textbf{Gaussian binomial coefficient}.

\begin{lemma}(The Standard Equations)\label{lem:st-eq} Let $\mU$ be an $\Fmk$ system and let $\Lambda_{r}$ be the set of all $r$-dimensional $\F_{q^m}$-subspaces of $\F_{q^m}^k$. We have
 \begin{equation}\label{eq:st-eq}
      \sum_{H\in\Lambda_r} |H\cap (\mU\setminus\{0\})|=(q^n-1)\binom{k-1}{r-1}_{q^m}.
 \end{equation}
\end{lemma}
\begin{proof}
Every element in $\mU\setminus\{0\}$ belongs to exactly $\binom{k-1}{r-1}_{q^m}$ $r$-dimensional subspaces in $\Lambda_r$. Therefore, \begin{equation*}
    \sum_{H\in\Lambda_r} |H\cap (\mU\setminus\{0\})| = \sum_{u\in\mU\setminus\{0\}}|\{H\in\Lambda \st u\in H\}|   =(q^n-1)\binom{k-1}{r-1}_{q^m},
\end{equation*}
which is the desired result.
\end{proof}

In the remainder of this section we describe the 1-to-1 correspondence between equivalence classes of nondegenerate
$\Fmkd$ codes and equivalence classes of $[n,k,d]_{q^m/q}$ systems.
We denote the set of equivalence classes of nondegenerate 
$\Fmkd$ codes by $\mC\Fmkd$, and the set of equivalence classes of $[n,k,d]_{q^m/q}$ systems by 
$\mU\Fmkd$. Next, we define a map
$$\Phi : \mC\Fmkd \to \mU\Fmkd$$
as follows: 
Given an equivalence class $[\mC] \in \mC\Fmkd$, let
$\Phi([\mC])$ be the equivalence class of the $\F_q$-span of the columns of a generator matrix of $\mC$. 
Vice versa, given an equivalence class $[\mU] \in \mU\Fmkd$, fix an $\F_q$-basis $\{g_1,\ldots,g_n\}$ of $\mU$ and let 
$\Psi([\mU])$ be the equivalence class of the code generated by the matrix having 
the $g_i$'s as columns.
In Theorem~\ref{thm:1-1} we will show that~$\Phi$ and~$\Psi$ are the inverse of each other.

We recall that the minimum rank distance of a code $\mC$ coincides with the minimum $\F_q$-dimension of the
linear space generated over $\F_q$ by the entries of $v\in\mC$. In particular,
$d^\rk(\mC) \le \dH(\mC)$. More precisely, the rank of a vector can be rewritten as
\begin{equation} \label{obvious}
    \rk(v) = \min \{\wH(vA) \st A \in \GL_n(q)\}.
\end{equation}
We will also repeatedly use the following simple fact: Let $V,H \subseteq W$ be nonzero finite dimensional vector spaces over $\F_q$ and let $\mB$ be the set of 
 $\F_q$-bases of $V$; then
\begin{equation} \label{eq:prel}
    \max\{|B \cap H| \st  B \in \mB\}=\dim(V \cap H).
\end{equation}

Finally, we will often use the following characterization of the rank of a vector.

\begin{lemma}\label{lem:rkvG}
 Let $\mC$ be a nondegenerate $\Fmk$   code and let $G$ be a generator matrix of $\mC$.  For any nonzero $v\in\F_{q^m}^k$ we have
 \begin{equation}\label{eq:rankvG}
     \rk(vG)=n-\dim_{\F_q}(\mU\cap\langle v\rangle ^\perp),
 \end{equation}
 where $\mU$ is the $\Fmk$ system generated by the $\Fq$-span of the columns of $G$.
\end{lemma}
\begin{proof}
Using~\eqref{obvious} we see that for all nonzero $v \in \Fm^k$ we have
    \begin{align*}
        \rk(vG) = \min\{\wH(vGA) \st A \in \GL_n(q)\} 
        = \min\{n-|\{i \st (GA)_i \in \langle v \rangle^\perp\}|\},
    \end{align*}
    where $(GA)_i$ is the $i$-th column of $GA$ and 
    $\langle v \rangle^\perp$ is the dual of the 1-dimensional code generated by $v$.
    As $A$ ranges over $\GL_n(q)$, the columns of $GA$ range over all bases of $\mU$.
    %the $\F_q$-columnspace of $G$. 
    Therefore we conclude by the identity in~\eqref{eq:prel}.
\end{proof}

The following result has already been shown in \cite{randrianarisoa2020geometric}. We include a complete proof in the Appendix.

\begin{theorem}\label{thm:1-1}
 The maps $\Phi$ and $\Psi$ are well-defined and 
 are the inverse of each other. In particular, they give a 1-to-1 correspondence between equivalence classes of nondegenerate $\Fmkd$ rank-metric codes and equivalence classes of $[n,k,d]_{q^m/q}$ systems.
 \end{theorem}

We also observe that combining Lemma \ref{lem:rkvG} with 
Remark \ref{rem:effectivelength} one obtains the following lower bound for the minimum distance of a rank-metric code.

\begin{corollary}\label{cor:lowerbounddistance}
 Let $\mC$ be an $\Fmkd$ code. Then
 $$ d\geq \dim_{\Fq}(\sigma^\rk(\mC))-(k-1)m. $$
\end{corollary}

As an application of Theorem \ref{thm:1-1}, we show that a nondegenerate rank-metric code always have a codeword of rank $\min\{n,m\}$. Note that this the largest possible rank a codeword can possibly have.

\begin{notation}
We denote by $\maxrk(\mC)$ the maximum rank of the codewords of a rank-metric code $\C \subseteq \F_{q^m}^n$. 
\end{notation}

\begin{proposition} \label{prop:max}
Let $\mC$ be a nondegenerate $[n,k]_{q^m/q}$ code, then $\maxrk(\mC)=\min\{n,m\}$. In particular,  if $n=m$ then an $[n,k]_{q^n/q}$ code is nondegenerate if and only if $\maxrk(\mC)=n$.
\end{proposition}

\begin{proof}
 If $\maxrk(\mC)=n$ then the statement is trivially true, so we may assume that $\maxrk(\mC)<n$. Let $\mU$ be any  $\Fmk$  system  associated with $\mC$ via Theorem \ref{thm:1-1}. 
 By Lemma~\ref{lem:rkvG} we have 
 that $\dim(H\cap \mU) \geq n-\maxrk(\mC)$ for each $\Fm$-hyperplane $H$ of $\Fm^k$.  Denote by~$\Lambda$ the set of all $\F_{q^m}$-hyperplanes of $\F_{q^m}^k$. Then we have
 $$
      (q^n-1)\binom{k-1}{1}_{q^m}=\sum_{H\in\Lambda} |H\cap (\mU\setminus\{0\})|\geq (q^{n-\maxrk(\mC)}-1)\binom{k}{1}_{q^m},
 $$
 where the first equality follows from Lemma \ref{lem:st-eq}. The above inequality is equivalent to
 $$ (q^n-1)(q^{(k-1)m}-1) \geq (q^{n-\maxrk(\mC)}-1)(q^{km}-1).$$
 Dividing both sides by $(q^{(k-1)m}-1)$, we obtain
 \begin{align*}
     q^n-1 & \geq (q^{n-\maxrk(\mC)}-1)\left(q^m+ \frac{q^m-1}{q^{(k-1)m}-1}\right) \\
&=q^{n+m-\maxrk(\mC)}-q^m+\frac{(q^{n-\maxrk(\mC)}-1)(q^m-1)}{q^{(k-1)m}-1} \\
&\geq q^{n+m-\maxrk(\mC)}-q^m. 
 \end{align*}
 Since $n-\maxrk(\mC)\ge 1$, this implies $m\leq \maxrk(\mC)$. Since, clearly, $\maxrk(\mC)\leq m$, then they must be equal.
\end{proof}

As an application of Proposition~\ref{prop:max}, we recover the characterization of optimal $\F_{q^m}$-linear anticodes given in \cite[Theorem 18]{ravagnani2016generalized} with a new and concise proof.

\begin{corollary}
 Let $\mC$ be an $\Fmk$ code with $k=\maxrk(\mC)$. If $m \ge n$, then $\mC$ has a basis made of vectors with entries in $\F_q$.
\end{corollary}
\begin{proof}
 We prove the result by induction on $n-k$. The case $n=k$ is immediate. Now assume that $n \ge k+1$ and that $\mC$ has $k=\maxrk(\mC)$. Fix a generator matrix $G$ for $\mC$. Since $k<n$, by Proposition \ref{prop:max} there exists $A \in \GL_n(q)$ such that the last column of $G \cdot A$ is zero. Denote by $G'$ the matrix obtained from $G \cdot A$ by deleting its last column. The code generated by $G'$ has $k=\maxrk(\mC)$ and therefore, by the induction hypothesis, has a basis made of vectors with entries in $\F_q$. This means that there exists $B \in \GL_k(q)$ such that $BG'$ (and thus $BGA$) has entries in $\F_q$. Therefore $BG=BGAA^{-1}$ has entries in $\F_q$ as well.
 \end{proof}

We conclude this subsection by surveying the connection 
between the generalized rank weights of an $\Fmk$ rank-metric code and any corresponding $\Fmk$ system. The definitions given here are equivalent to those of \cite{randrianarisoa2020geometric}. 
We denote the set of  Frobenius-closed subspaces of $\Fm^n$ by $\Lambda_q(n,m)$, that is,
$$ \Lambda_q(n,m):=\left\{ \mV\leq \Fm^n \st \theta(\mV)=\mV \right\},$$
where $\theta:x \longmapsto x^q$ is the $q$-Frobenius automorphism in $\Fm$ (extended component-wise to vectors). It is known that $\Lambda_q(n,m)$ corresponds to the set of subspaces of $\Fm^n$ that have a basis of vectors in $\Fq^n$; see \cite[Theorem 1]{giorgetti2010galois}.
\begin{definition} \label{def:rkw}
 Let $\mC$ be an $\Fmk$ code. For every $r=1,\dots,k$, the \textbf{$r$-th generalized rank weight} of $\mC$ is the integer
 $$\drk_r(\mC):=\min\left\{\dim(\mV) \st \mV \in \Lambda_q(n,m), \ \dim(\mV\cap \mC)\geq r\right\}. $$
\end{definition}

The following result was shown in \cite{randrianarisoa2020geometric}. We state it here for completeness and give a proof in the Appendix. 
 
 \begin{theorem}\label{thm:genrankweight}
Let $\mC$ be an $\Fmkd$ nondegenerate code and let $\mU$ be any $[n,k,d]_{q^m/q}$ system associated to $\mC$. For any $r=1,\dots, k$ the $r$-th generalized rank weight is given by
 \begin{align*}
 \drk_r(C) &= n- \max \left\{\dim_{\Fq}(\mU\cap H) \st H \mbox{ is an } \Fm\mbox{-subspace of codim. } r  \mbox{ of } \Fm^k \right\} \\
  &= \min \left\{\dim_{\Fq}(\mU + H) \st H \mbox{ is an } \Fm\mbox{-subspace of codim. } r  \mbox{ of } \Fm^k\right\}-m(k-r).
 \end{align*}
 In particular, the minimum rank distance of $\mC$ is given by 
$$d=n- \max \{\dim_{\Fq}(\mU\cap H) \st H \mbox{ is an } \Fm\mbox{-hyperplane of }  \Fm^k \}.$$
\end{theorem}

\subsection{Simplex and One-Weight Codes in the Rank Metric} 

In this subsection we use the geometric approach on rank-metric codes to define simplex codes as the natural counterpart of simplex Hamming-metric codes. In particular, this allows to characterize one-weight codes in the rank metric, recovering the results of~\cite{randrianarisoa2020geometric} in this context.

\begin{lemma}\label{lem:tpowers}
 Let $a,b,c,d$ be positive integers such that $a\leq b$ and $c\leq d$, and let $t\geq2$ be an integer. Suppose that $(t^a-1)(t^b-1)=(t^c-1)(t^d-1)$. Then $a=c$ and $b=d$.
\end{lemma}

\begin{proof}
 By contradiciton, assume that $(a,b) \neq (c,d)$. Moreover, without loss of generality we can assume $a \leq c$. Since $(a,b) \neq (c,d)$, then  we need to have  $a < c\leq d$ (if $a=c$ clearly also $b=d$). Moreover, we also have that $b>a$, otherwise the equality is not possible.
 By expanding the equality $(t^a-1)(t^b-1)-(t^c-1)(t^d-1)=0$, and dividing by $t^a$, we get 
 $$ t^{b}-t^{b-a}-t^{c+d-a}+t^{c-a}+t^{d-a}-1=0.$$
 All the exponents of $t$ appearing above  are  positive integers, hence we get a contradiction, since the left hand side is equal to $-1 \mod t$.
\end{proof}

\begin{proposition}\label{prop:simplexcode}
 Let $k\geq 2$, let $\mC$ be a $[km,k]_{q^m/q}$ code, and let $G$ be a generator matrix of~$\mC$. The following are equivalent.
 \begin{enumerate}
     \item $\mC$ is nondegenerate.
     \item The $\Fq$-span of the columns of $G$ is $\Fm^k$.
     \item $\mC$ is a one-weight code (with minimum distance $m$).
     \item $\drk(\mC^\perp)>1$.
     \item $\drk(\mC^\perp)=2$.
     \item $\mC$ is linearly equivalent to a code whose generator matrix is
     \begin{equation}\label{eq:standardformsimplex} \left( \begin{array}{c|c|c|c}I_k & \alpha I_k & \cdots & \alpha^{m-1}I_k \end{array}\right),\end{equation}
     where $\alpha \in \Fm$ satisfies $\Fm=\Fq(\alpha)$.
 \end{enumerate}
\end{proposition}

\begin{proof}
\underline{$(1) \Rightarrow (2)$}: If $\mC$ is nondegenerate, then its support has dimension $km$, which is also the dimension of the 
 associated $[km,k]_{q^m/q}$ system. 

\underline{$(2) \Rightarrow (6)$}:  The code $\mC$ has effective length $km$ ands $\mU=\Fm^k$ as corresponding $\Fmk$ system. Hence, $\mU$ has a basis given by $\mB=\{\alpha^ie_j \st 0\leq i \leq m-1, 0\leq j \leq k-1  \}$. Thus, $\mC$ belongs to the same equivalence class of the code whose generator matrix is \eqref{eq:standardformsimplex}.

\underline{$(6) \Rightarrow (5)$}: Without loss of generality, we can assume that $\mC$ is the code whose generator matrix $G$ is \eqref{eq:standardformsimplex}. Since $\mC$  is nondegenerate, by Proposition~\ref{prop:newnondegeneracy} we have $\drk(\mC^\perp)>1$. Moreover, $G$ is a parity check matrix for $\mC^\perp$ and from that  it is easy to see that the vector $v=\alpha e_1-e_{k+1}$ belongs to $\mC^\perp$ and  has rank weight 2. Thus $\drk(\mC^\perp)=2$.

\underline{$(5) \Rightarrow (4)$}: Clear.

\underline{$(4) \Rightarrow (1)$}: The equivalence between $(4)$ and $(1)$ holds for every rank-metric code, by Proposition \ref{prop:newnondegeneracy}.

\underline{$(2) \Rightarrow (3)$}: Let $\mC$ be the $\Fmk$ code generated by $G$. By hypothesis,  the $\Fmk$ system corresponding to $\mC$ is $\mU=\Fm^k$. Moreover, for every nonzero $v\in\Fm^k$, by \eqref{eq:rankvG} it holds that $$\rk(vG)=km-\dim_{\Fq}(\mU\cap \langle v\rangle^\perp)=km-(k-1)m=m.$$ 

\underline{$(3) \Rightarrow (1)$}: Let $\mC$ be a $[km,k]_{q^m/q}$ code. Let $n\leq km$ be its effective length, that is, $n=\dim(\sigma^\rk(\mC))$. This means that $\mC$ can be isometrically embedded in $\Fm^n$, obtaining a code~$\mC'$.
Then $\mC'$ is a nondegenerate $\Fmk$ code with the same weight distribution as $\mC$. In particular, $\mC'$ is a one-weight code as well. Fix a generator matrix for $\mC'$ and consider the associated $\Fmk$ system, which we call $\mU$. Since $\mC'$ is a one-weight code,  we have $|H\cap (\mU\setminus\{0\})|=(q^a-1)$ for every $\Fm$-hyperplane of $\Fm^k$.  Therefore, if we denote by $\Lambda$ the set of all the $\Fm$-hyperplanes in $\Fm^k$, we have 
$$ \sum_{H\in \Lambda} |H\cap (\mU\setminus\{0\})|=\binom{k}{1}_{q^m}(q^a-1).$$
Moreover, by applying Equation \eqref{eq:st-eq} to the right-hand side, we obtain 
$$ (q^{km}-1)(q^a-1)=(q^{(k-1)m}-1)(q^n-1).$$
By Lemma \ref{lem:tpowers}, we have $a=(k-1)m$ and $n=km$. Hence $\mC$ itself is nondegenerate.
\end{proof}

We call \textbf{simplex rank-metric code}
a code that satisfies
any of the equivalent conditions in Proposition \ref{prop:simplexcode}.
Note that Proposition \ref{prop:simplexcode} also implies the following characterization of one-weight codes in the rank metric, which is the analogue of the main result of~\cite{bonisoli1983}.

\begin{corollary}[Classification of one-weight rank-metric codes]\label{cor:classificationsimplex}
Let $k \ge 2$ and let $\mC$ be an $\Fmkd$ one-weight code. Then, the effective length of $\mC$ is $km$ and $d=m$. That is, $\mC$ is isometric to a simplex rank-metric $[km,k,m]_{q^m/q}$ code.
\end{corollary}

\begin{proof}
 If $n\leq km$, as shown in the proof of Proposition \ref{prop:simplexcode}, it has to be $n=km$ and $\mC$ is a simplex rank-metric code. Assume now $n>km$. Since the effective length of an $\Fmk$ is always at most $km$, then we can isometrically embed $\mC$ in a $[km,k]_{q^m/q}$ code $\mC^\prime$, with the same weight distribution. By Proposition \ref{prop:simplexcode}, $\mC^\prime$ has to be a simplex rank-metric code.
\end{proof}

We remark that there is a strong analogy between simplex rank-metric codes and their homonyms in the Hamming metric, which is confirmed by both their weight distributions and by geometric characterization. 

Indeed, by Corollary \ref{cor:classificationsimplex}, simplex rank-metric codes are the only nondegenerate one-weight codes in the rank-metric, just like simplex codes in the Hamming metric, up to repetition. In fact, 
simplex codes in the Hamming metric are the only 
{projective} one-weight codes (where \textbf{projective} means that no two columns of one, and thus any, generator matrix are linearly dependent\label{pagedefproj}).

From a geometric point of view, simplex codes in the Hamming metric have a generator matrix whose columns are formed by all the points of $\PG(k-1,q)$. In the rank-metric, simplex codes are associated 
to the $[km,k]_{q^m/q}$ system $\Fm^k$, which is the natural analogue in the rank metric.

We conclude by observing that a definition of simplex code in the rank metric has been recently proposed in \cite{martinez2020hamming} (the definition has been given for sum-rank-metric codes, which specialize to rank-metric codes by taking a single matrix block).
The simplex codes defined in \cite{martinez2020hamming} are different from the simplex codes considered in this paper. For example, one can check they are not one-weight in general. From a geometric viewpoint, the definition of simplex code proposed in this paper appears therefore more natural.

%%%%%%%%%%%%%%%%%%%%%%%

\section{From Rank-Metric to Hamming-Metric Codes}\label{sec:rank-hamming}
\label{sec:4}

In this section we explore various connections between codes in the rank and in the Hamming metric. In particular, we show how to construct a Hamming-metric code from a rank-metric one and describe how the parameters of the two codes relate to each other. 
%We start by introducing linear sets in finite geometry.

\subsection{Linear Sets}

Linear sets in finite geometry can be viewed as a generalizations of subgeometries. Their name was first proposed by Lunardon in \cite{lunardon1999normal}, where linear sets are used for special constructions of blocking sets. The very first example of linear set 
is probably due to Brouwer and Wilbrink; see \cite{brouwer1982blocking}. The interested reader is referred to \cite{polverino2010linear} for an in-depth treatment of linear sets. 

A special family of linear sets, which is of particular interest for this paper, is the one of scattered linear sets introduced by Blokhuis and Lavrauw in \cite{blokhuis2000scattered}. Recently, Sheekey and Van de Voorde observed a connection between scattered linear sets and rank-metric codes with optimal parameters in \cite{sheekey2016new,Sheekey2020}; see \cite{polverino2020connections} for a survey on this topic.

\begin{definition}
Let $\mU$ be an $\Fmk$ system. The $\F_q$-\textbf{linear set} in $\PG(k-1, q^m)$ of rank $n$ associated to $\mU$ is the set $$L_\mU:=\{\langle u \rangle_{\F_{q^m}} \st u\in \mU\setminus\{0\}\},$$
where $\langle u \rangle_{\F_{q^m}}$ denotes the projective point corresponding to $u$.
\end{definition}

Let  $\Lambda=\PG(W,\Fm)$ be the projective subspace corresponding to the $\Fm$-subspace~$W$ of~$\Fm^k$. We define the \textbf{weight} of $\Lambda$ in $L_\mU$ as the integer
$$ \wt_{\mU}(\Lambda):=\dim_{\Fq}(\mU\cap W).$$
If $\Lambda$ is an hyperplane, that  is, if $\Lambda=\PG(W,\Fm)$ with $W=\langle v\rangle^\perp$ for some nonzero $v\in\Fm^k$, then $\wt_{\mU}(\Lambda)=n-\rk(vG)$,  where $G$ is a $k\times n$ matrix associated to $\mU$; see Lemma \ref{lem:rkvG}.
Observe moreover that for a point $P\in \PG(k-1,q^m)$ we have that $P \in L_\mU$ if and only if $\wt_{\mU}(P)\geq 1$. 

\begin{remark}
The original definition of linear sets does not assume the space $\mU$ to be a $\Fmk$ system, i.e., that $\langle \mU\rangle_{\Fm}$ is the whole space $\Fm^k$. However, if $\dim_{\Fm}(\langle \mU\rangle_{\Fm})=k-i$, one can assume up to equivalence that $\mU\subseteq \langle e_1,\ldots,e_{k-i}\rangle_{\Fm}=:V$, and then study $\mU$ in the projective subspace $\PG(k-i-1,q^m)$ induced by $V$. 
\end{remark}

For any $\Fmk$ system $\mU$, the cardinality of the associated linear set $L_\mU$ satisfies 
\begin{equation}\label{eq:linear_set} |L_\mU|\leq \frac{q^n-1}{q-1}. \end{equation}
A linear set $L_\mU$ whose cardinality meets \eqref{eq:linear_set} with equality is said to be \textbf{scattered}. Equivalently, a linear set $L_{\mU}$ is scattered if and only if $\wt_{\mU}(P)=1$ for each $P\in L_{\mU}$. 
We also observe that \eqref{eq:linear_set} can be refined as follows.

\begin{lemma}\label{lem:linear_set_multiplicity}
 Let $\mU$ be an $\Fmk$  system. Then
 $$\sum_{P \in \PG(k-1,q^m)}\frac{q^{\wt_{\mU}(P)}-1}{q-1}=\frac{q^n-1}{q-1}. $$
\end{lemma}

\begin{proof}
Let $\Lambda_1$ be the set of $1$-dimensional $\Fm$-subspaces of $\Fm^k$. Then, we have
 \begin{align*} \sum_{P \in \PG(k-1,q^m)} \hspace{-0.1cm} \frac{q^{\wt_{\mU}(P)}-1}{q-1}&=\frac{1}{q-1} \sum_{V \in \Lambda_1} (q^{\dim_{\Fq}(\mU\cap V)}-1) =\frac{1}{q-1} \sum_{V \in \Lambda_1} |V\cap(\mU\setminus \{0\})| =\frac{q^n-1}{q-1},
 \end{align*} \nopagebreak
 where the latter equality follows from Lemma \ref{lem:st-eq}.
\end{proof}

\subsection{The Associated Hamming-Metric Code}
\label{sub:assoc}

The notion of a linear set allows us to describe a connection between rank-metric codes and some particular  codes in the Hamming metric. This connection was also  observed in \cite{sheekey2019scatterd}. For a $\Fmk$ system $\mU$ and a point $P \in \PG(k-1,q^m)$, define
$$ \mm_{\mU}(P):=\frac{q^{\wt_{\mU}(P)}-1}{q-1}.$$
The identity of Lemma \ref{lem:linear_set_multiplicity} can be written as
\begin{equation}\label{eq:length_multiplicity}\sum_{P \in \PG(k-1,q^m)}\mm_{\mU}(P)=\frac{q^n-1}{q-1}.\end{equation}
Denote by $\mU(n,k)_{q^m/q}$ the set of $\Fmk$ system and by
$\mP(n,k)_{q^m}$ the set of projective  $[n,k]_{q^m}$ systems. Define the map
$$\begin{array}{rccc}  &\mU(n,k)_{q^m/q} & \longrightarrow & \mP(\frac{q^n-1}{q-1},k)_{q^m}, \\ & \mU & \longmapsto & (L_{\mU},\mm_\mU),\end{array}$$
where $(L_{\mU},\mm_\mU)$ denotes the multiset $L_{\mU}$ with multiplicity function $\mm_\mU$. The parameters $\frac{q^n-1}{q-1}$ and $k$ of the projective system $(L_\mU,\mm_\mU)$ directly follow from \eqref{eq:length_multiplicity}.
It is easy to see that this map is compatible with the equivalence relations on  $\mU(n,k)_{q^m/q}$ and on $\mP(\frac{q^n-1}{q-1},k)_{q^m}$. Indeed, the actions defining the equivalence classes are given in both cases by the group $\mathrm{PGL}(k,q^m)$. We thus constructed a map
$$\begin{array}{rccc} \Ext^\HH: &\mU[n,k]_{q^m/q} & \longrightarrow & \mP[\frac{q^n-1}{q-1},k]_{q^m}, \\ 
%& \mU & \longmapsto & (L_{\mU},\mm_\mU),
\end{array}$$
where ${\mU[n,k]_{q^m/q}}$ and ${\mP[\frac{q^n-1}{q-1},k]_{q^m}}$ denote  the set of equivalence classes of $\Fmk$ systems and the set of equivalence classes of projective ${[\frac{q^n-1}{q-1},k]_{q^m}}$ systems, respectively. This maps leaves also the parameter $d$ of the projective ${[\frac{q^n-1}{q-1},k]_{q^m}}$ system fixed, as the following result shows.

\begin{lemma}\label{lem:min_distance_associated_H}
 Let $[\mU]$ be the equivalence class of $[n,k,d]_{q^m/q}$ systems. Then $[(L_{\mU},\mm_\mU)]$ is the equivalence class of a projective $[\frac{q^{n}-1}{q-1},k,\frac{q^n-q^{n-d}}{q-1}]_{q^m}$ system. In other words, the map 
 $$\Ext^\HH: \mU[n,k,d]_{q^m/q}  \longrightarrow  \mP\bigg[\frac{q^n-1}{q-1},k,\frac{q^n-q^{n-d}}{q-1}\bigg]_{q^m}$$
 is well-defined.
\end{lemma}

\begin{proof}
 The fact that the map $\Ext^\HH$ sends equivalence classes of $\Fmk$ systems in equivalence classes of projective $[\frac{q^n-1}{q-1},k]_q$ systems has already been observed above. We only need to show the compatibility between the third parameters. 
 More precisely, we need to show that for a given $[n,k,d]_{q^m/q}$ system $\mU$, every element in $\Ext^\HH([\mU])$ is a projective $[\frac{q^n-1}{q-1},k,\frac{q^n-q^{n-d}}{q-1}]_q$ system. Fix the projective $[\frac{q^n-1}{q-1},k,d']_q$ system $(L_\mU,\mm_\mU)$, and denote by $\Lambda_{k-1}$ the set of $\Fm$-hyperplanes of $\Fm^k$. Then for any $H\in \Lambda_{k-1}$ we have
  \begin{align*} \sum_{P \in \PG(H,\Fm)} \mm_\mU(P) 
 &=\sum_{P \in \PG(H,\Fm)} \frac{q^{\wt_{\mU}(P)}-1}{q-1} \\
 %&=\frac{1}{q-1}\sum_{\substack{V\le H\\ \dim_{\Fm}(V)=1}} (q^{\dim_{\Fq}(V\cap\mU)}-1) \\
 &=\frac{1}{q-1}\sum_{\substack{V\subseteq H\\ \dim_{\Fm}(V)=1}} |V\cap(\mU\setminus\{0\}|\\
 &=\frac{1}{q-1}|H\cup (\mU\setminus\{0\})|\\
 &=\frac{q^{\dim_{\Fq}(H\cap \mU)}-1}{q-1},
 \end{align*}
 where the second to last identity follows from the fact that  $\{V\setminus \{0\}: V\subseteq H, \dim_{\Fm}(V)=1\}$ is a partition of $H\setminus \{0\}$. Therefore we obtain
 \begin{equation*}
     d'=\frac{q^n-1}{q-1}-\max\bigg\{\frac{q^{\dim_{\Fq}(H\cap \mU)}-1}{q-1} \st H \in \Lambda_{k-1} \bigg\}=\frac{q^n-1}{q-1}-\frac{q^{n-d}-1}{q-1}=\frac{q^n-q^{n-d}}{q-1}. \qedhere
 \end{equation*} 
\end{proof}

\begin{definition}
 Let $\mC$ be a nondegenerate $[n,k,d]_{q^m/q}$ rank-metric code. 
 We will call any Hamming-metric code
 in $(\Psi^\HH\circ\Ext^\HH\circ\Phi)([\mC])$
 \textbf{associated} with $\mC$. Note that any such an object is a
  $[\frac{q^n-1}{q-1},k,\frac{q^n-q^{n-d}}{q-1}]_{q^m}$ code.
\end{definition}

The Hamming-metric code associated to $\mC$ in the previous definition
is clearly not unique. 
However, the choice of the code is irrelevant when focusing on properties that are invariant under monomial equivalence. Therefore, for ease of notation, in the sequel we denote by $\mC^\HH$ any code that belongs to $(\Psi^\HH\circ\Ext^\HH\circ\Phi)([\mC])$.

\begin{example}
Let $q=2$, $n=4$ and $m=3$. Consider $\F_{8}=\F_2[\alpha]$, where $\alpha^3+\alpha+1=0$.
Moreover, let $\C$ be the $[4,2,1]_{8/2}$ code whose generator matrix
$$G= \begin{pmatrix}
1 & 0 & 0 & 0 \\
0 & 1 & \alpha & \alpha^2
\end{pmatrix}.$$ 
Take the $[4,2,1]_{8/2}$ system $\mU$ spanned by the columns of $G$, i.e.,
$\mU=\{(a,\beta): a\in \F_2, \beta \in \F_8\}$. The weights of the points in $\PG(1,8)$ with respect to $\mU$ are given by
\begin{align*}\wt_{\mU}([1:a])&=1, \qquad \mbox{ for every } a \in \F_8\\
\wt_{\mU}([0:1])&=3.
\end{align*}
Hence, we obtain that $\Ext^\HH(\mU)=(\PG(1,8),\mm_\mU)$, where
\begin{align*}\mm_\mU([1:a])&=1, \qquad \mbox{ for every } a \in \F_8\\
\mm_\mU([0:1])&=7.
\end{align*}
At this point, any code $\mC^{\HH}=C\in(\Psi^\HH\circ\Ext^\HH\circ\Phi)([\mC])$ is monomially equivalent to  the $[15,2,8]_{8}$ (Hamming-metric) code whose generator matrix is
$$ G_{\Ext}=\left(\begin{array}{ccccccccccccccc} 1 & 1 & 1 & 1 & 1 & 1 & 1 & 1 & 0 & 0 & 0 & 0 & 0 & 0 & 0 \\
 0 & 1 &  \alpha &  \alpha^2 &  \alpha^3 & \alpha^4 & \alpha^5 & \alpha^6 & 1 & 1 & 1 & 1 & 1 & 1 & 1
\end{array}\right).$$
\end{example}

\begin{example}[Simplex Rank-Metric Code] Take $\C$ to be the $[km,k,m]_{q^m/q}$ simplex rank-metric code, whose corresponding $[km,k,m]_{q^m/q}$ system is $\Phi([\C])=[\Fm^k]$. Denote $\mU:=\Fm^k$ and consider the weight of each point $P\in\PG(k-1,q^m)$ in $L_{\mU}$. For $P=[v]$, we have
$$\wt_{\mU}(P)=\dim_{\Fq}(\mU\cap \langle v\rangle_{\Fm})=\dim_{\Fq}(\langle v\rangle_{\Fm})=m.$$
Therefore by applying the map $\Ext^{\HH}$ we obtain
$$\Ext^{\HH}([\mU])=[(L_{\mU},\mm_{\mU})],$$
where $L_{\mU}=\PG(k-1,q^m)$ and
$$\mm_{\mU}(P)=\frac{q^m-1}{q-1} \quad \mbox{for all } P \in \PG(k-1,q^m).$$
In particular, any code in $(\Psi^\HH\circ\Ext^\HH\circ\Phi)([\mC])$ is monomially equivalent to  the concatenation of $\frac{q^m-1}{q-1}$ copies of the  $[\frac{q^{km}-1}{q^m-1},k,q^{(k-1)m} ]_{q^m}$ simplex code in the Hamming metric.
\end{example}

Lemma \ref{lem:min_distance_associated_H} shows how the
fundamental parameters of a nondegenerate $[n,k,d]_{q^m/q}$ rank-metric code $\mC$ relate to 
those of an associated Hamming-metric code $\mC^\HH$. 
The connection can be made even more precise. For example, we can say how the weight distributions of the two codes relate to each other.

\begin{theorem}\label{thm:weightdistribution_H_R}
 Let $\mC$ be a nondegenerate $[n,k,d]_{q^m/q}$ rank-metric code with rank-weight distribution $\{A_i^\rk(\mC)\}_{i}$. Then the Hamming-weight distribution of $\C^\HH$ is $\{A_j^\HH(\mC^\HH)\}_j$ with
 $$ A_j^\HH(\mC^\HH)=\begin{cases} A_i^\rk(\mC) & \mbox{ if } j=\frac{q^n-q^{n-i}}{q-1},\\
 0 & \mbox{ otherwise. }\end{cases}$$
\end{theorem}

\begin{proof}
 Let $G$ be a generator matrix for $\C$ and denote by 
 $\mU$ the $\Fq$-span of its columns. Let~$G_{\Ext}$ be a generator matrix for $\C^{\HH}$ whose columns are the elements of the multiset $(L_{\mU},\mm_{\mU})$.
 Doing the same computations as in Lemma \ref{lem:min_distance_associated_H} we obtain that, 
 for every $u\in \Fm^k\setminus\{0\}$, 
 \begin{equation}\label{eq:weightcorresp}
 \wt^{\HH}(uG_{\Ext})=\frac{q^n-1}{q-1}-\sum_{P \in \PG(H_u,\Fm)} \frac{q^{\wt_{\mU}(P)}-1}{q-1}=\frac{q^n-q^{n-\rk(uG)}}{q-1},
 \end{equation}
 where $H_u:=\langle u\rangle^\perp$.
\end{proof}

\begin{remark} \label{rmk:MW}
 While the connection between the dual of a code $\mC$ and the dual of $\mC^\HH$ seems to be difficult to describe explicitly, we remark that their weight distributions (in the rank and Hamming metric, respectively) are linked via the theory of MacWilliams identities; see~\cite{macwilliams1977theory} for a general reference. 
 More precisely, the Hamming weight distribution of $(\mC^\HH)^\perp$ can be written in terms  of the Hamming weight distribution of $\mC^\HH$. By Theorem~\ref{thm:weightdistribution_H_R}, the latter can be written in terms of the rank weight distribution of $\mC$ which, in turn, can be expressed in terms of the rank weight distribution of~$\mC^\perp$. We do not go into the details of the computation.
\end{remark}

\begin{remark}
Theorem \ref{thm:weightdistribution_H_R} generalizes various known results on Hamming-metric codes obtained from linear sets. This is the case of the two-weight Hamming-metric codes arising from maximum scattered linear sets found by Blokhuis and Lavrauw in 
\cite[Section 5]{blokhuis2000scattered}, and of the  Hamming-metric codes with $h+1$ weights recently presented by Zini and Zullo in  \cite[Theorem~7.1]{zini2021scattered}.
\end{remark}

Finally, one can also prove the following result connecting the generalized weights of~$\mC$ and~$\mC^\HH$ (in the respective metrics). These code invariants can be found in Definition~\ref{def:rkw} and in the Appendix, respectively.
Since the argument is very similar to that in the proof of Theorem~\ref{thm:weightdistribution_H_R}, the details are omitted.

\begin{theorem}%[Correspondence of the Generalized Weights]
 Let $\mC$ be a nondegenerate $[n,k,d]_{q^m/q}$ rank-metric code with generalized rank-weights $\{ \drk_i(\mC)\}_{i}$. Then the generalized Hamming-weights of $\C^\HH$ are given by $\{\dH_i(\mC^\HH)\}_i$, where
 $$ \dH_i(\mC^\HH)=\frac{q^n-q^{n-\drk_i(\mC)}}{q-1}. $$
\end{theorem}

\subsection{The Total Weight of a Rank-Metric Code}

In this subsection we continue comparing codes in the rank and in the Hamming metric. Our focus is 
on the rank-metric analogue of a fundamental parameter of a Hamming-metric code, namely, its \textit{total weight}. It is well-known that the latter only depends on the field size and on the code's dimension and effective length. More precisely, if $\mC \subseteq \F_{q}^n$ is a Hamming-nondegenerate code, then
\begin{equation} \label{totH}
    \sum_{v \in \mC} \wt^{\HH}(v) = n(q^k-q^{k-1}).
\end{equation}
This simple result, which has numerous applications in classical coding theory (for example, a simple proof of the Plotkin bound for linear codes), does not have an immediate analogue in the rank metric. Indeed, it is easy to find examples of rank-nondegenerate codes having the same parameters $(q,m,n,k)$ but for which the quantity
$\sum_{v \in \mC} \rk(v)$ is not a  constant.

In this section, we argue that, in the ``total weight''  context, a convenient analogue of $\wt^{\HH}(v)$ is $q^{n-\rk(v)}$. We start by recalling the following $q$-analogue of the Pless identities; see~\cite{huffman2010fundamentals}.

\begin{notation}
For a prime power $q$ and integers $n,m,k,j,r$, let
 $$f_q(n,m,k,j,r):= \sum_{\nu=j}^r q^{m(k-\nu)} \qbin{n-j}{\nu-j}{q} \qbin{r}{\nu}{q} \,  \prod_{\ell=0}^{\nu-1} (q^\nu-q^\ell).$$
\end{notation}

\begin{theorem}[Theorem 30 of \cite{de2018weight}] \label{PPP}
 Let $\mC$ be an $[n,k,d]_{q^m/q}$ code. Then for all $0 \le r \le n$ we have
 $$\sum_{v \in \mC} q^{r(n-\rk(v))} = \sum_{j=0}^r A_j(\mC^\perp) \, f_q(n,m,k,j,r).$$
\end{theorem}

In analogy with Remark~\ref{rmk:MW}, we observe that a different statement of Pless-type identities can in principle be obtained by combining the correspondence $\mC \to \mC^\HH$ with the classical Pless identities for Hamming-metric codes. For the purposes of this section, Theorem~\ref{PPP} is what we will need.

In this paper, we are not only interested in the $q$-analogue of the total weight of a code, but also in other related quantities. In order to unify their treatment, it is convenient to regard the Hamming/rank weight of the nonzero elements of a code as a discrete random variable, which we simply denote by  $\mC^*$, 
$\EE^\rk$ and
$\Var^\rk$ for the mean and variance of  (a function of) $\mC^*$,  viewed as a random variable in the sense explained above. 

In the sequel, we call a code $\mC \subseteq \F_{q^m}^n$ \textbf{rank-2-nondegenerate} if $\drk(\mC^\perp) \ge 3$. Codes with this property are the rank-metric analogues of projective codes in the Hamming metric; see page~\pageref{pagedefproj}.
The previous theorem has the following simple consequences.

\begin{corollary}
\label{kor}
 Let $\mC$ be an $[n,k,d]_{q^m/q}$ code. If $\mC$ is rank-nondegenerate, then
 \begin{align*}
%  \EE^\rk[q^{n-\mC}] &=  1+ q^{-m} (q^n-1), \\
 \EE^\rk[q^{n-\mC^*}] &=  \frac{-q^n + q^{mk} + q^{m(k-1)}(q^n-1)}{q^{mk}-1}, \\
%   \Var^\rk[q^{n-\mC}] &\ge  q^{-mk}f_q(n,m,k,0,2) -  \EE^\rk[q^{n-\mC}]^2, \\
  \Var^\rk[q^{n-\mC^*}] &\ge
  \frac{-q^{2n}+f_q(n,m,k,0,2)}{q^{mk}-1} - \EE^\rk[q^{n-\mC^*}]^2,
  \end{align*}
   where the latter  lower bound is attained with equality if and only if $\mC$ is rank-2-nondegenerate.
\end{corollary}

Corollary~\ref{kor} establishes the rank-metric analogue of the formula for the total weight~of a Hamming-metric code in~\eqref{totH}.
It also shows that, for a rank-2-nondegenerate code $\mC$, the variance of the random variable $q^{n-\mC^*}$ only depends on a few code's parameters. While the formulas in Corollary~\ref{kor} are quite involved and not immediate to interpret, their asymptotics as $q \to +\infty$ can be explicitly computed. The estimates describe how the variance behaves over large fields.

\begin{proposition}
 Let $\mC$ be an $[n,k,d]_{q^m/q}$ code. If $\mC$ is rank-nondegenerate then $n \le km$ and, as $q \to +\infty$,
 $$\EE^\rk[q^{n-\mC^*}] \sim \begin{cases}
1 & \mbox{if $n \le m-1$,}\\
q^{n-m} & \mbox{if $m+1 \le n \le km$,} \\
2 & \mbox{if $n=m$}.
\end{cases}$$
If in addition $\mC$ is rank-2-nondegenerate, then $n \le mk/2$ and for $k \ge 3$ and $q \to +\infty$ we have
$$
\Var^\rk[q^{n-\mC^*}] \sim \begin{cases}
q^{-m+n+1} & \mbox{if $k \le n \le m-2$ or $m+2 \le n \le mk/2$,} \\
1  & \mbox{if $n=m-1$,} \\
q   & \mbox{if $n=m$,} \\
q^2 & \mbox{if $n=m+1$.}
\end{cases}
$$
\end{proposition}
\begin{proof} The first part of the statement easily follows from Proposition~\ref{prop:nleqkm} and Corollary~\ref{kor}. To prove the second part, 
we start by applying the rank-metric Singleton bound~\cite{delsarte1978bilinear,gabidulin1985theory} to $\mC^\perp$, obtaining $m(n-k) \le n(m-\drk(\mC^\perp)+1)
\le n(m-2)$. This implies $n \le mk/2$, as desired.

We now turn to the asymptotic estimates.
To simplify the notation, write 
 $f_q$ instead of $f_q(n,m,k,0,2)$. Lengthy computations show that  
$$f_q=q^{mk}+q^{m(k-1)}(q^n-1)(q+1) + q^{m(k-2)+1}(q^n-1)(q^{n-1}-1).$$
Therefore 
$$f_q(n,m,k,0,2) \sim \begin{cases}
q^{mk} & \mbox{if $n \le m-2$,} \\
q^{mk+1} & \mbox{if $n = m$,}\\
q^{mk+2n-2m} & \mbox{if $n \ge m+2$,} \\
2q^{mk} & \mbox{if $n=m-1$,} \\
2q^{mk+2} & \mbox{if $n=m+1$.}
\end{cases}$$
From the first part of the statement we also have
$$\EE^\rk[q^{n-\mC^*}]^2 \sim \begin{cases}
1 & \mbox{if $n \le m-1$,}\\
q^{2n-2m} & \mbox{if $n \ge m+1$,} \\
4 & \mbox{if $n=m$}.
\end{cases}$$
Using $k \ge 3$ (needed in the case $n=m+1$), this easily gives the asymptotics of
$$\Var^\rk[q^{n-\mC^*}] = \frac{-q^{2n}+f_q(n,m,k,0,2)}{q^{mk}-1} - \EE^\rk[q^{n-\mC^*}]^2$$
for $n=m-1$, $n=m$, and $n=m+1$. To compute the asymptotics in the other cases, write
$$\frac{-q^{2n}+f_q(n,m,k,0,2)}{q^{mk}-1} - \EE^\rk[q^{n-\mC^*}]^2
= \frac{A_q - B_q}{(q^{mk}-1)^2},$$
where $A_q=(q^{mk}-1)(-q^{2n}+f_q)$ and 
$B_q=(-q^n+q^{mk}+q^{m(k-1)}(q^n-1))^2$. 

If $n \le m-2$ then $f_q \sim q^{mk}+q^{m(k-1)+n+1}$. Therefore $A_q \sim q^{2mk}+q^{m(2k-1)+n+1}$ and 
$B_q \sim -q^{2mk}+2q^{m(2k-1)+n}$, from which the desired asymptotic estimate follows.

If $m+2 \le n \le mk/2$, then $m+2 \le n \le m(k-1)$, because $k \ge 3$. We then have $f_q \sim q^{m(k-2)+2n} + q^{m(k-1)+n+1}$ and thus
$A_q \sim q^{m(2k-2)+2n} + q^{m(2k-1)+n+1}$, $B_q \sim q^{2m(k-1)+2n} + 2q^{m(2k-1)+n}$. This again implies the desired asymptotic estimate.

\end{proof}

\bigskip

%%%%%%%%%%%%%%%%%%%%%%%

\section{Minimal Rank-Metric Codes: Geometry and Properties}
\label{sec:5}

The next two sections of this paper are devoted to the theory of minimal codes in the rank metric.
In this first section we propose a definition of minimal and establish a 1-1 correspondence between $\Fmk$ minimal rank-metric codes and $\Fmk$ systems. This allows us to investigate the main properties of this new family of codes.

\begin{definition}\label{def:mincodes}
Let $\C$ be an $\Fmk$ code. A codeword $v\in\mC$ is a \textbf{minimal codeword} if, for every $v'\in \mC$, $\sigma^\rk(v') \subseteq \sigma^\rk(v)$  implies $v'=\alpha v$ for some $\alpha \in \F_{q^m}$.
We say that~$\mC$ is \textbf{minimal} if all its codewords are minimal.
\end{definition}

\begin{lemma}\label{lem:supp}
 Let $v \in \Fm^n$. The following hold.
 \begin{enumerate}
     \item There exists $A \in \GL_n(q)$ such that $\sigma^\rk(vA)=\langle e_i \st i \in \sH(vA) \rangle$.
     \item Let $I\subseteq \{1,\dots,n\}$. Then,  $\sigma^\rk(v)\subseteq \langle e_i \st i \in I \rangle$ if and only if $I \supseteq \sH(v)$. In particular, $$\sH(v)=\arg\min\{|I| \st \sigma^\rk(v)\subseteq \mE_I \},$$ where $\mE_I:=\langle e_i \st i \in I\rangle$ and $\sigma^\rk(v) \subseteq \langle e_i \st i \in \sH(v) \rangle.$
 \end{enumerate}
\end{lemma}

\begin{proof}
\begin{enumerate}
     \item Let  $r=\dim(\sigma^\rk(v))$. By Proposition \ref{prop:prel}, there exist a matrix $A$  and a basis $\Gamma$ of $\Fm/\Fq$, such that $\Gamma(vA)$ is in Smith normal form. Hence, $\sH(vA)=\{1,\ldots,r\}$ and $\sigma^\rk(vA)=\langle e_i \st i \in \{1,\ldots, r\} \rangle$.
    \item Let $I=\sH(v)$ and fix any basis $\Gamma$ of $\Fm/\Fq$. The rows indexed by $\{1,\dots,n\}\setminus I$ in $\Gamma(v)$ are identically zero. Hence, $\sigma^\rk(v)\subseteq \langle e_i \st i \in I \rangle$.
    %$\sigma^\rk(v)\subseteq \langle e_i \mid i \in I \rangle$. 
    Vice versa, assume that there exists $t \in \sH(v)\setminus I$. Fix any basis $\Gamma$ of $\Fm/\Fq$. Since $t \in \sH(v)$, there exists $j\in [m]$ such that $\Gamma(v_t)_j\neq 0$. Hence, the vector $a=(\Gamma(v_1)_j,\ldots,\Gamma(v_n)_j)$ belongs to $\sigma^\rk(v)$ and has a nonzero entry in the $t$-th coordinate. Thus, $\sigma^\rk(v)\not\subseteq \langle e_i \st i \in I \rangle$. The second statement immediately follows. \qedhere
\end{enumerate}
\end{proof}

\subsection{Linear Cutting Blocking Sets and the Parameters of Minimal Codes}

In this subsection we give a geometric characterization of minimal codes in the rank metric. This will allow us to derive bounds on their parameters. 

We start with the  $q$-analogue of the notion of a cutting blocking set.
%, introduced in \cite{bonini2020minimal}.  

\begin{definition} 
 A $\Fmk$ system $\mU$ is called a \textbf{linear cutting blocking set} if for  any $\Fm$-hyperplanes $H,H^\prime \subseteq \Fm^k$ we have $(\mU\cap H)\subseteq (\mU\cap H')$ implies $H=H'$.
 We will say that that 
 $\mU$ is a linear cutting $\Fmk$ blocking set to emphasize the parameters.
\end{definition}

While the term ``linear cutting blocking set'' might seem not fully consistent with the terminology used so far (since such an object is not a linear set), one can verify 
that an $\Fmk$ system $\mU$ is a linear cutting blocking set if and only if its associated linear set $L_{\mU}$ is a cutting blocking set in $\PG(k-1,q^m)$. The proof of this fact can be found in Section \ref{sec:HM_RM}; see Theorem~\ref{thm:HM_iff_RM}. This explains the choice of the terminology.

We will need the following characterization of linear cutting blocking sets.

\begin{proposition}\label{prop:cutting}
 A $\Fmk$  system $\mU$ is a linear cutting blocking set if and only if for every $\Fm$-hyperplane $H$ we have $\langle H\cap \mU\rangle_{\Fm}=H$.
\end{proposition}

\begin{proof}
$(\Leftarrow)$ Let $H, H'$ be two $\Fm$-hyperplanes of $\Fm^k$ such that $(\mU\cap H)\subseteq (\mU\cap H')$. Hence, $H=\langle H\cap \mU\rangle_{\Fm}\subseteq \langle H'\cap \mU\rangle_{\Fm}=H'$. Since $H$ and $H'$ are both hyperplanes, they have to be equal. 

$(\Rightarrow)$ Suppose by contradiction that there exists an $\Fm$-hyperplane $H$ of $\Fm^k$ such that $\langle H\cap \mU\rangle_{\Fm}=X \subsetneq H$. Then, for every hyperplane $H'\supset X$ we have   $(\mU\cap H)\subseteq (\mU\cap H')$. Since there are at least $q^m$ such hyperplanes different from $H$, we obtain that $\mU$ is not a linear cutting blocking set.
\end{proof}

\begin{corollary}\label{cor:sizecutting}
 If $\mU$ is a linear cutting $\Fmk$ blocking set, then for every $\Fm$-hyperplane of $\Fm^k$ we have $|H\cap \mU|\geq q^{k-1}$.
\end{corollary}

\begin{proof}
 Let $t:=\dim_{\Fq}(H\cap \mU)$. Then an $\Fq$-basis for $H\cap \mU$ is also a set of $\Fm$-generators for $\langle H\cap \mU\rangle_{\Fm}$. Hence, since $\mU$ is a linear cutting blocking set, by Proposition \ref{prop:cutting} we have
 $$m(k-1)= \dim_{\Fq}(\langle H\cap \mU\rangle_{\Fm})\leq mt,$$
 which shows that $t\geq k-1$. 
\end{proof}

The geometric description of minimal rank-metric codes via linear cutting blocking sets relies on the following characterization of the inclusion of rank supports.

\begin{theorem}\label{thm:revInclusion}
 Let $G$ be a generator matrix for a nondegenerate $\Fmk$ code, $\mU$ be the corresponding $\Fmk$ system and $u,v \in \Fm^k\setminus\{0\}$. Then, $$\sigma^\rk(uG)\subseteq \sigma^\rk(vG) \qquad \mbox{if and only if} \qquad (\langle u\rangle^\perp \cap \mU) \supseteq (\langle v\rangle^\perp \cap \mU).$$
\end{theorem}
\begin{proof}

 \underline{$(\Leftarrow)$} Let $x_1,\ldots, x_t$ be an $\Fq$-basis of the space $X:=(\langle v\rangle^\perp \cap \mU)$. Let $A \in \GL_n(q)$ be such that 
 $$GA=(\,x_1\,|\,\cdots \, |\,x_t\,|\, G' \, ),$$
 where $G'\in \Fm^{k\times(n-t)}$. We have $uGA=(0,\ldots,0| uG')$ and $vGA=(0,\ldots,0 | vG')$. Moreover, by~\eqref{eq:rankvG} we have $\rk(vGA)=n-t$ and, by Lemma \ref{lem:supp}, $\sigma^\rk(vGA)=\langle e_i \st i=t+1\ldots, n\rangle$ and $\sigma^\rk(uGA)\subseteq \langle e_i \st i=t+1\ldots, n\rangle$. This means that $\sigma^\rk(uGA)\subseteq \sigma^\rk(vGA)$. Finally, Proposition~\ref{prop:prel} implies  $\sigma^\rk(uG)\subseteq \sigma^\rk(vG)$.
 
 \underline{$(\Rightarrow)$} Assume now that $\sigma^\rk(uG)\subseteq \sigma^\rk(vG)$. Let $r:=\rk(vG)$. By the first part of Lemma \ref{lem:supp} there exists $A\in\GL_n(q)$ such that $\sigma^\rk(vGA)=\langle e_1,\ldots,e_r\rangle$. Hence, $\sigma^\rk(uGA)\subseteq \sigma^\rk(vGA)=\langle e_1,\ldots,e_r\rangle$. Denote by $x_1,\ldots, x_n$ the columns of $GA$, which also form a basis of $\mU$. In this notation we have $\langle v\rangle^\perp \cap \mU=\langle x_{r+1},\ldots, x_n \rangle_{\Fq}$. Moreover, by the second part of Lemma \ref{lem:supp} we have $\sH(uGA)\subseteq \{1,\ldots,r\}$. This implies that $x_i\in\langle u\rangle^\perp$ for $i=r+1,\ldots, n$. Hence, $(\langle u\rangle^\perp \cap \mU)\supseteq(\langle v\rangle^\perp \cap \mU)$.
\end{proof}

By combining Theorem \ref{thm:revInclusion} and the correspondence stated in Theorem~\ref{thm:1-1} we obtain the following.

\begin{corollary}\label{coro:1-1cutting}
 The correspondence $(\Phi,\Psi)$ defined in Section \ref{sec:geo-rk} induces a 1-1 correspondence between minimal rank-metric codes and linear cutting blocking sets.
\end{corollary}

 Corollary \ref{coro:1-1cutting} has several consequences in the theory of minimal codes. The first result we drive concerns the construction of new minimal codes from existing ones.  

\begin{corollary}
Let $\mC$ be an $\Fmk$ minimal rank-metric code with  generator matrix $G$, and let $v \in \Fm^k$. Then the  $[n+1,k]_{q^m/q}$ code $\bar{\mC}=\rs(G \, \mid \, v^\top)$ is minimal.
\end{corollary}

\begin{proof}
Without loss of generality, we may assume that $\mC$ is nondegenerate. Let $\mU$ be any $\Fmk$ system associated to $[\mC]$ and let $\bar{\mU}=\langle \mU, v \rangle_{\F_q}$. If $v \in \mU$, then by Proposition \ref{prop:newnondegeneracy} the code $\bar{\mC}$ is degenerate and it is equivalent to the code $\{(\, c \,\mid \, 0 \,)  \st c \in \mC\}$, which is clearly minimal. 
Hence, assume that $v \notin \mU$. By Proposition \ref{prop:newnondegeneracy} we have that  $\bar{\mC}$ is nondegenerate and~$\bar{\mU}$ is an $[n+1,k]_{q^m/q}$ system associated to $\bar{\mC}$. 
Let $H$ be any $\F_{q^m}$-hyperplane of $\Fm^k$. Then
$$ H \supseteq  \langle H\cap \bar{\mU} \rangle_{\F_{q^m}} =  \langle H\cap (\mU + \langle v \rangle_{\F_q}) \rangle_{\F_{q^m}}\supseteq  \langle H\cap \mU\rangle  = H,$$
where the latter equality follows from the fact that, since $\C$ is minimal,
$\mU$ is a linear cutting blocking set by Theorem \ref{coro:1-1cutting}. Therefore $\bar{\mU}$ is also a linear cutting blocking set and we conclude using Theorem \ref{coro:1-1cutting} again.
\end{proof}

The following two results are also consequences of Corollary~\ref{coro:1-1cutting} and provide information about the parameters of a minimal $\Fmk$ code.

\begin{corollary}\label{cor:max_rank}
 Let $\mC$ be a  minimal $\Fmk$ code. Then for every $c\in \mC$ we have $\rk(c)\leq \dim_{\Fq}(\sigma^\rk(\C))-k+1$. In particular, $\maxrk(\mC) \le \dim_{\Fq}(\sigma^\rk(\C))-k+1\leq n-k+1$.
\end{corollary}

\begin{proof}
Let $n'=\dim_{\Fq}(\sigma^\rk(\C))$ for ease of notation. As observed in Remark \ref{rem:effectivelength}, we can isometrically embed $\C$ in $\Fm^{n'}$. Moreover, the resulting code is minimal if and only if $\C$ is minimal. Therefore we can assume without loss of generality that $\mC$ is nondegenerate of length $n=\dim_{\Fq}(\sigma^\rk(\C))$.
 Let $\mU$ be any $\Fmk$  system associated to $\mC$. By Corollary~\ref{coro:1-1cutting}, $\mU$ is a linear cutting blocking set. From the  proof of Corollary \ref{cor:sizecutting} we get that $\dim_{\Fq}(H\cap \mU)\geq k-1$ for every $\Fm$-hyperplane of $\Fm^k$, and we conclude using Lemma \ref{lem:rkvG}.
\end{proof}

\begin{corollary}\label{coro:lowerboundlength}
 If $\mC$ is a  minimal $\Fmk$ code with $k \geq 2$, then $n \geq k+m-1$.
\end{corollary}

\begin{proof}
Without loss of generality we shall assume that $\mC$ is nondegenerate. Therefore by Proposition \ref{prop:max} we have $\maxrk(\mC)=\min\{m,n\}$. Since $k \ge 2$, by Corollary \ref{cor:max_rank}  we also have $\maxrk(\mC) \le n-k+1<n$. Therefore
$\maxrk(\mC)=m$ and using again the fact that $\maxrk(\mC) \le n-k+1$ we find $n \ge m+k-1$, as desired.
\end{proof}

\subsection{Connections with Hamming-Metric Minimal Codes}\label{sec:HM_RM}

It is natural to ask how the notions of minimality in the rank and in the Hamming metric relate to each other. This is the question we address in this subsection. In particular, we prove that a 
nondegenerate rank-metric code $\mC$ is minimal if and only if its associated code(s) $\mC^\HH$ is minimal; see Section~\ref{sub:assoc} for the notation.

The following result shows that minimality in the Hamming metric implies minimality in the rank metric. We propose two proofs, one in coding theory parlance and the other in the language of projective systems.

 \begin{proposition} \label{prop:HMnew}
 Let $\mC$ be an $\Fmk$ code with the property of being  
 Hamming-minimal. Then $\mC$ is rank-minimal.
\end{proposition}

\begin{proof}
 Suppose that $\mC$ is not a minimal rank-metric code. Then there exist two codewords $v,v^\prime$ that are $\Fm$-linearly independent such that $\sigma^\rk(v)\subseteq \sigma^\rk(v^\prime)$. By Lemma \ref{lem:supp}, we also have that $\sH(v^\prime)=\arg\min\{|I| \st \sigma^\rk(v^\prime)\subseteq \mE_I \}$, where $\mE_I:=\langle e_i \st i \in I\rangle$. Since $\sigma^\rk(v)\subseteq \sigma^\rk(v^\prime)$, we have $\sH(v)\subseteq\sH(v^\prime)$, and therefore $\mC$ is not Hamming-minimal.
\end{proof}

\begin{proof}[Second proof]
 Without loss of generality, we may assume that $\mC$ is nondegenerate. Let $G$ be a generator matrix for $\mC$, and let $\mB$ be the basis of the associated $\Fmk$ system $\mU$ formed by the columns of $G$. Then $\mM:=\{\langle u\rangle_{q^m} \st u\in\mB\}$ is a projective $[n,k]_{q^m}$ system in $\PG(k-1,q^m)$. By Hamming-minimality and Theorem \ref{thm:cuttingHammingminimal}, it is a cutting blocking set. Hence $\langle \PG(H,\Fm)\cap \mM\rangle=\PG(H,\Fm)$ for every $\Fm$-hyperplane $H$ of $\Fm^k$. Let $H$ be an $\Fm$-hyperplane  of $\Fm^k$ and let 
 $$V:=\langle H\cap \mU \rangle=\langle H\cap \langle\mB\rangle_{\Fq} \rangle.$$
 Then $$\PG(V,\Fm)=\langle \PG(H,\Fm) \cap L_{\mU}\rangle\supseteq\langle \PG(H,\Fm)\cap \mM\rangle=\PG(H,\Fm),$$
 showing that $V=H$,
 We conclude by applying Proposition \ref{prop:cutting}.
\end{proof}

\begin{remark}\label{rem:rankM_notHM}
 The converse of Proposition \ref{prop:HMnew} is false in general.
 For example, let $(q,m,n)=(2,3,4)$. Write $\F_{8}=\F_2[\alpha]$, where $\alpha^3+\alpha+1=0$.
The code generated by
$$G= \begin{pmatrix}
1 & 0 & 0 & 0 \\
0 & 1 & \alpha & \alpha^2
\end{pmatrix}$$ is rank-minimal but not Hamming-minimal.  Moreover, the code $\mC\cdot A$ is not Hamming-minimal for any $A\in \GL_3(2)$. Indeed, if this was the case, then there would exist a Hamming-minimal~$[4,2]_8$ code, which contradicts~\cite[Theorem 2.14]{alfarano2020three}.
\end{remark}

The previous results and examples show that 
minimality in the rank and in the Hamming metric gives rise to very different concepts. We now show that the correspondence $\mC \to \mC^\HH$ is more natural in this context, as it translates rank-minimality precisely into Hamming-minimality.

 \begin{theorem}\label{thm:HM_iff_RM}
   Let $\mC$ be a nondegenerate $[n,k,d]_{q^m/q}$ rank-metric code. Then $\mC$ is minimal if and only if $\C^{\HH}$ is Hamming-minimal.
 \end{theorem}

\begin{proof}
 By Corollary~\ref{coro:1-1cutting}, $\mC$ is minimal if and only if any   $\Fmk$ system $\mU$ associated with $\mC$ is a linear cutting blocking set.  Now, consider the linear set $L_{\mU}$. We show that $\mU$ is a linear cutting blocking set if and only if $L_{\mU}$ is a cutting blocking set in $\PG(k-1,q^m)$. Let $H$ be an $\Fm$-hyperplane of $\Fm^k$, then $L_{\mU}\cap \PG(H,\Fm)=L_{\mU}\cap L_H=L_{\mU \cap H}$, and hence
 $$ \langle L_{\mU}\cap \PG(H,\Fm) \rangle =\langle L_{\mU}\cap L_H\rangle=\langle L_{\mU \cap H} \rangle. $$
Moreover,  for every subset $S\subseteq\Fm^k$, one has $\smash{L_{\langle S\rangle_{\Fm}}=\langle L_S \rangle}$. This implies that $L_{\mU}$ is a cutting blocking set in $\PG(k-1,q^m)$ if and only if for every $\Fm$-hyperplane $H$ of $\Fm^k$ we have $\smash{L_{\langle \mU \cap H\rangle_{\Fm}} =L_{H}}$. Since $H$ and $\langle \mU \cap H\rangle_{\Fm}$ are both $\Fm$-linear, the linear set that they define coincide with the respective projective subspaces. Therefore $L_{\mU}$ is a cutting blocking set in $\PG(k-1,q^m)$ if and only if $\langle H\cap \mU\rangle_{\Fm}=H$ for every $\Fm$-hyperplane $H$ in $\Fm^k$, as claimed. We conclude using Theorem \ref{thm:cuttingHammingminimal}  -- which states that a linear code is Hamming-minimal if and only if the associated projective system is a cutting blocking set -- and observing that, by definition,~$L_{\mU}$ is the projective system associated to $\C^{\HH}$.
\end{proof}

Theorem \ref{thm:HM_iff_RM} allows us to transfer results known for minimal codes in the Hamming metric to the rank metric setting. 
For example, the following is the rank-metric analogue of the characterization in \cite[Theorem 11]{MR3857591}.

\begin{theorem}
Let $\mathcal{C}$ be an $\Fmk$ code. Then $\mathcal{C}$ is minimal if and only if
\[\sum_{\lambda\in {\F_{q^m}} \setminus \{0\}}
q^{-\rk(c+\lambda c')} \neq  (q^m-1)\cdot q^{-\rk(c)}-q^{-\rk(c')}+1\]
for all linearly independent $c,c'\in \mathcal{C}$.
\end{theorem}

\begin{proof}
By Theorem \ref{thm:HM_iff_RM}, $\mathcal{C}$ is rank-minimal if and only if any associated-Hamming metric code~$\C^{\HH}$ is Hamming-minimal. We can now conclude by using
 \cite[Theorem 11]{MR3857591} and \eqref{eq:weightcorresp}.
\end{proof}

\begin{remark}
 It is natural to ask if the best known criterion for Hamming-minimality, namely the \textit{Ashikhmin-Barg condition} of \cite[Lemma 2.1]{ashikhmin1998minimal}, can be transferred to the rank-metric context. 
 The mentioned result states that  every $[n,k,d]_{q^m}$ code satisfying $w_{\max}(q^m-1)<q^md$ is Hamming-minimal, where $w_{\max}$ denotes the maximum Hamming weight of a codeword.

 One may naturally try to use Ashikhmin-Barg condition together with Theorem \ref{thm:HM_iff_RM} and Theorem \ref{thm:weightdistribution_H_R} to  obtain a sufficient condition rank-minimality. This can be done as follows. 
 
 Let~$\C$ be a nondegenerate $\Fmkd$ code.
 By Corollary~\ref{coro:lowerboundlength}, we may assume without loss of generality that $n \geq m$. By Proposition \ref{prop:max},  the maximum rank of a codeword in~$\C$ is~$m$. Now consider the associated Hamming-metric code $\C^{\HH}$. Using Theorem  \ref{thm:weightdistribution_H_R} we see that the minimum distance of $\C^{\HH}$ is $(q^n-q^{n-d})(q-1)^{-1}$ and that the maximum Hamming weight of a codeword in $\C^{\HH}$ is  $(q^n-q^{n-m})(q-1)^{-1}$. Therefore imposing the Ashikhmin-Barg condition yields the following:
 A nondegenerate $\Fmkd$ code is rank-minimal
 if  
 \begin{equation}\label{eq:AB_rank} (q^n-q^{n-m})(q^m-1)<q^m(q^n-q^{n-d}).
 \end{equation}
 However, it is not difficult to see that \eqref{eq:AB_rank} is only satisfied when $d=m$, that is, when $\C$ is the $[km,k,m]_{q^m/q}$ simplex code; see Proposition~\ref{prop:simplexcode}. In other words, the rank metric analogue of the Ashikhmin-Barg condition is trivial. 
\end{remark}

\bigskip

%%%%%%%%%%%%%%%%%%%%%%%

\section{Minimal Rank-Metric Codes: Existence and Constructions}
\label{sec:6}

In this second section on minimal rank-metric codes we turn to their existence and constructions. In particular, in the light of the geometric characterization of Corollary~\ref{coro:1-1cutting} and of the lower bound of Corollary~\ref{coro:lowerboundlength}, we investigate the existence of short minimal codes.
We start by showing some simple examples of minimal codes. Then we construct a family of 3-dimensional minimal codes using scattered linear sets, and establish the existence of minimal rank-metric codes for all $n \ge 2k+m-2$ using a counting argument. The last part of this section is devoted to a new parameter of rank-metric codes, which we call the \textit{linearity index} and use to investigate further the structure of minimal codes.

\subsection{First Examples of Minimal Rank-Metric Codes}
A natural question is whether a simplex rank-metric code is minimal or not. Indeed, 
in the Hamming-metric simplex codes are among the simplest and best known minimal codes.

\begin{theorem}
 Let $\mC$ be a $[km,k,m]_{q^m/q}$ simplex rank-metric code. Then $\mC$ is minimal.
\end{theorem}

\begin{proof}
 By the definition,
 any $[km,k]_{q^m/q}$ system associated to 
 $\mC$ is $\Fm^k$; see Proposition~\ref{prop:simplexcode}. The latter is clearly a  linear cutting blocking set, since $H\cap \Fm^k=H$ for each $\Fm$-hyperplane~$H$ of $\Fm^k$.
\end{proof}

The following criterion is a sufficiency result to have a minimal rank-metric code.

\begin{proposition}\label{prop:minimalk-1m+1}
Let $\mC$ be a nondegenerate $\Fmk$ code with $n\geq (k-1)m+1$. Then~$\mC$ is minimal.
\end{proposition}

\begin{proof}
 Let  $\mU$ be any  $\Fmk$ system corresponding to $\mC$ (up to equivalence) and let $H$ be an $\Fm$-hyperplane of $\Fm^k$. By Proposition \ref{prop:cutting}, we need to show that $\langle H\cap \mU\rangle_{\Fm}=H$. Since $H$ is also an $\Fq$-space,  we can compute the $\Fq$-dimension of $H\cap \mU$ as follows:
 \begin{align*} \dim_{\Fq}(H\cap \mU) & = \dim_{\Fq}(H)+\dim_{\Fq}(\mU)-\dim_{\Fq}(H+\mU) \\
 &= (k-1)m+n-\dim_{\Fq}(H+\mU) \\
 &\geq (k-1)m+(k-1)m+1-km \\
 &=(k-2)m+1.
 \end{align*}
 This implies that $\langle H\cap \mU\rangle_{\Fm}$ has $\Fm$-dimension strictly greater than $k-2$ and since it is  contained in $H$, it has to be equal to $H$. 
\end{proof}

Proposition \ref{prop:minimalk-1m+1} shows that every nondegenerate $[n,2]_{q^m/q}$ code with $n=m+1$ is minimal. This means that the bound of Corollary~\ref{coro:lowerboundlength} is sharp  for $k=2$. It is natural to ask if
the bound is sharp
for other values of $k$. We will show in Section \ref{sec:three_dim_minimal} that this happens also for $k=3$.

\subsection{Three-Dimensional Minimal Rank-Metric Codes}\label{sec:three_dim_minimal}
In this section we study minimal $[n,3]_{q^m/q}$ codes. In particular we prove that they exist for every $n\geq m+2$ under the assumption that $m\geq 4$. This also implies that for $k=3$ and $m \geq 4$ the bound of Corollary~\ref{coro:lowerboundlength} is sharp. 

The first result that we provide links the existence of scattered linear sets with $3$-dimensional minimal rank-metric codes.

\begin{theorem}\label{thm:scattered_implies_minimal}
 Let $\C$ be a nondegenerate $[n,3]_{q^m/q}$ code with $n \geq m+2$ and let  $\mU$ be any $[n,3]_{q^m/q}$ system corresponding to $\C$.  If $L_{\mU}$ is a scattered linear set, then $\C$ is a minimal rank-metric code.
\end{theorem}

\begin{proof}
 Let $\C^{\HH}\in(\Psi^\HH\circ\Ext^\HH\circ\Phi)([\mC]) $ be any  Hamming-metric code associated with $\C$. By Theorem \ref{thm:HM_iff_RM}, $\C$ is rank-minimal if and only if  $\C^{\HH}$ is Hamming-minimal, which is in turn equivalent to the fact that $L_{\mU}$ is a cutting blocking set in $\PG(2,q^m)$. Consider now the multiplicity function associated to $L_{\mU}$ in the projective $[\frac{q^n-1}{q-1},k]_{q^m}$ system $\Ext^\HH(\mU)$. Since $L_{\mU}$ is scattered, this means that every point of $L_{\mU}$ has multiplicity $1$. Let $G$ be any generator matrix of $\C^{\HH}$, and let $v\in\Fm^3\setminus\{0\}$. Since by Proposition \ref{prop:max} the maximum rank of a codeword in $\C$ is $m$, using Theorem \ref{thm:weightdistribution_H_R} we get
 $$ \wt^{\HH}(vG)\leq \frac{q^{n}-q^{n-m}}{q-1}.$$
 Thus,
 \begin{align*}|L_\mU \cap \langle v\rangle^{\perp}|&=\frac{q^n-1}{q-1}-\wt^{\HH}(vG) 
  \geq \frac{q^{n-m}-1}{q-1} \geq q+1.
 \end{align*}
 In particular, $L_{\mU}$ is a $(q+1)$-fold blocking set, and in $\PG(2,q^m)$ this also implies that $L_{\mU}$ is cutting. 
\end{proof}

Thanks to Theorem  \ref{thm:scattered_implies_minimal}, 
the existence of 
certain minimal rank-metric codes reduces to the existence of certain scattered linear sets.
There is a well known upper bound on the parameters of these objects, due to Blokhuis and Lavrauw; see \cite{blokhuis2000scattered}. If $\mU$ is a  $[n,k]_{q^m/q}$ system such that $L_{\mU}$ is scattered, then
\begin{equation}\label{eq:upper_bound_scattered}
    n\leq \frac{km}{2}.
\end{equation}
In this context, much progress has been made in the study of \textit{maximum scattered linear sets}, which are linear sets whose parameters meet the bound in \eqref{eq:upper_bound_scattered}  with equality. A construction of such linear sets was first provided by Blokhuis and Lavrauw for $k$ even; see~\cite{blokhuis2000scattered}.
When instead $k$ is odd and $m$ is even, a construction of linear sets meeting~\eqref{eq:upper_bound_scattered} for infinitely many parameters was given by Bartoli, Giulietti, Marino and Polverino in \cite[Theorem~1.2]{bartoli2018maximum}. The picture was then completed by Csajb\'ok, Marino, Polverino and Zullo; see~\cite{csajbok2017maximum}.

\begin{theorem}[see \textnormal{\cite[Theorem 2.4]{csajbok2017maximum}}]\label{thm:max_scattrered}
 Assume that $km$ is even. Then there exists a $[\frac{km}{2},k]_{q^m/q}$ system such that $L_{\mU}$ is scattered. 
\end{theorem}

When $km$ is odd, not much is known yet. One of the few existence results on the maximum rank of a scattered linear set is the following, due to Blokhuis and Lavrauw.

\begin{theorem}[see \textnormal{\cite[Theorem 4.4]{blokhuis2000scattered}}]\label{thm:existence_scattered}
 Let $k,m$ be positive integers and $q$ be a prime power. There  exists an $[ab,k]_{q^m/q}$ system such that $L_{\mU}$ is scattered, whenever $a$ divides $k$, $\gcd(a,m)=1$ and
 $$ ab<\begin{cases}\frac{km-k+3}{2} & \mbox{ if } q=2 \mbox{ and } a=1,\\
 \frac{km-k+a+3}{2} & \mbox{ otherwise.}
 \end{cases}$$
\end{theorem}

 In contrast with the most common line of research in the theory of scattered linear set, in this paper we are primarily interested in short nondegenerate minimal codes, and thus in linear sets with small rank. For this reason, we state the following  simple lemma, whose proof is omitted.

\begin{lemma}\label{lem:scattered_n_n-1}
 Let $\mU$ be an $[n,k]_{q^m/q}$ system such that $L_{\mU}$ is a scattered linear set. If $n>k$, then there exists an $[n-1,k]_{q^m/q}$ system $\mV\subseteq \mU$ such that $L_{\mV}$ is scattered.
\end{lemma}

%\begin{proof}
%\textcolor{red}{ trivial} \ALE{Non la metto, ok?}
%\end{proof}

We conclude this subsection by combining the previous three results with each other. This yields the following existence theorem for $3$-dimensional minimal rank-metric codes.

\begin{theorem}
  Suppose that $m\not\equiv 3,5 \mod 6$ and $m\ge 4$. Then there exists a (nondegenerate) minimal $[m+2,3]_{q^m/q}$ code.
\end{theorem}

\begin{proof}
Observe that by Theorem \ref{thm:scattered_implies_minimal} it is enough to prove that there exists an $[m+2,3]_{q^m/q}$ system $\mU$ such that $L_{\mU}$ is scattered. 

 First, assume that $m$ is even. Then, by Theorem \ref{thm:max_scattrered}, we have that there exists a  $[\frac{3m}{2},3]_{q^m/q}$ system such that $L_{\mU}$ is scattered. Then, since $m+2\le \frac{3m}{2}$ whenever $m \ge 4$, using Lemma \ref{lem:scattered_n_n-1} we obtain the desired $[m+2,3]_{q^m/q}$ system.
 
 Now assume that $m$ is odd and $m\not\equiv 0 \mod 3$. Write $m=3s+i$. We use Theorem~\ref{thm:existence_scattered} with $a=3$ and $b=s+1$, which shows the existence of an 
$[\frac{m+3-i}{2},3]_{q^m/q}$ system $\mU$ such that~$L_{\mU}$ is scattered. If $m \equiv 1 \mod 3$, we get the desired result.
 \end{proof}

\begin{remark}
 In the remaining cases, finding scattered linear sets of rank $m+2$ in $\PG(2,q^m)$ seems in general a difficult task. For instance, when $m=5$, the existence of a $[7,3]_{q^5/q}$ system $\mU$ defining a scattered linear set  was recently shown in \cite[Theorem 5.1]{bartoli2020evasive}, but only in characteristic~$2$, $3$ and $5$ and under some  restriction on the field size.  
\end{remark}

We illustrate the construction of  minimal codes based on Theorem \ref{thm:scattered_implies_minimal} with an explicit  example.
\begin{example}
 We describe a construction of a $[6,3]_{q^4/q}$ system $\mU$ such that $L_{\mU}$ is a scattered linear set, which was proposed in \cite{BALL2000294}. First,  consider the finite field $\F_{q^{12}}=\F_{q^4}(\delta)=\Fq(\eta)$ and identify it with $\F_{q^4}^3$ by fixing the $\F_{q^4}$-basis $\{1,\delta,\delta^2\}$. Choose $\alpha, \beta \in \F_{q^{12}}$ such that 
 $$\begin{cases}
  \beta^{q^7+q^4-q^3+q}\neq -(\beta^{q^3}-\beta^{q^6+q^3+1})^{q-1}, \\
  \alpha^{q^3+1}=\beta^{q^3}-\beta^{q^6+q^3+1}\neq 0, \\
  \beta^{q^9+q^6+q^3+1}=1.
 \end{cases}
 $$
 Then the set 
 $$\mU=\left\{\gamma \in \F_{q^{12}} \st \gamma^{q^6}+\alpha\gamma^3+\beta=0 \right\}$$
 is a  $[6,3]_{q^4/q}$ system $\mU$ such that $L_{\mU}$ is a scattered linear set.
 
 More concretely, take $q=2$ and $\eta$ such that $\eta^{12}+\eta^7+\eta^6+\eta^5+\eta^3+\eta+1=0$. One can check that $\alpha=\eta^{64}$ and $\beta=\eta^7$ satisfy the properties above. We then have 
 $$\mU=\left\{\gamma \in \F_{2^{12}} \st \gamma^{64}+\alpha\gamma^3+\beta=0 \right\}=\langle \eta^{6}, \eta^{22}, \eta^{63}, \eta^{89}, \eta^{166}, \eta^{289} \rangle_{\F_2}.$$
 If we write $\F_{2^{12}}=\F_{16}(\eta)$ and $\F_{16}=\{0\} \cup \{\lambda^i \st 0 \leq i \leq 14\}$, where
 $\lambda=\eta^{273}$ satisfies  $\lambda^4+\lambda+1=0$, then we can express the generators of $\mU$ above
 in coordinates with respect to the $\F_{16}$-basis $\{1,\eta,\eta^2\}$ of $\F_{2^{12}}$. One of the  $[6,3]_{16/2}$ codes in $\Phi([\mU])$ is generated by the matrix
 $$G= \begin{pmatrix} \lambda^{4} & \lambda^{10} & \lambda^{8} & \lambda^{3} & \lambda^{9} & \lambda^{7} \\
 \lambda^{14} & \lambda^{8} & \lambda & \lambda^{8} & 0 & \lambda^{8} \\
 \lambda^{10} & 0 & \lambda^{6} & \lambda^{5} & \lambda^{11} & \lambda^{3} \end{pmatrix}.$$
 By Theorem \ref{thm:scattered_implies_minimal}, this code is minimal.
\end{example}

\subsection{Existence Results for Minimal Rank-Metric Codes}

In this subsection we establish a general existence result for minimal rank-metric codes. We prove that minimal rank-metric codes exist 
for all parameter sets $(n,m,k)$ with 
$m \ge 2$ and $n \ge 2k+m-2$ (and any $q$). Combining this with previous results, we then give parameter intervals
for which nondegenerate minimal codes exist and do not exist.

\begin{lemma} \label{lem:exist}
Let $m$, $n$, $k$ be positive integers and suppose 
$n \ge k \ge 2$. If
\begin{equation}\label{eq:existence}
    \frac{(q^{mn}-1)(q^{m(n-1)}-1)}{(q^{mk}-1)(q^{m(k-1)}-1)}- \frac{1}{2} \, \sum_{i=2}^m \frac{1}{q^m-1} \qbin{m}{i}{q} \, \prod_{j=0}^{i-1} (q^n-q^j) \left( \frac{q^{mi}-1}{q^m-1}-1  \right)
\end{equation}
is positive, then there exists a minimal $[n,k]_{q^m/q}$ code.
\end{lemma}

\begin{proof} 
We use an argument inspired by the methods of~\cite{gruica2020common} but which is simpler and avoids the graph theory language.
Form a set of representatives for the equivalence classes of nonzero vectors in $\F_{q^m}^n$. Call this set $\mQ$ and let 
$$\mP=\{P=\{x,y\} \subseteq \mQ \st x \neq y, \, \sigma^\rk(x) \subseteq \sigma^\rk(y) \mbox{ or } \sigma^\rk(y) \subseteq \sigma^\rk(x)\}.$$
The $[n,k]_{q^m/q}$ non-minimal codes are the $k$-dimensional subspaces $\mC \subseteq \F_{q^m}^n$ such that $P \subseteq \mC$ for some $P \in \mP$. Their number is at most
$$\sum_{P \in \mP} |\{\mC \subseteq \F_{q^m}^n \st \mC \supseteq P\}| = |\mP| \qbin{n-2}{k-2}{q}.$$
Therefore, the minimal $[n,k]_{q^m/q}$ codes are at least
$$\qbin{n}{k}{q^m} - |\mP| \qbin{n-2}{k-2}{q^m} = 
\qbin{n-2}{k-2}{q^m} \left( \frac{(q^{mn}-1)(q^{m(n-1)}-1)}{(q^{mk}-1)(q^{m(k-1)}-1)} - |\mP| \right).$$
In particular, a minimal $[n,k]_{q^m/q}$ code
exists if
\begin{equation} \label{kk}
    \frac{(q^{mn}-1)(q^{m(n-1)}-1)}{(q^{mk}-1)(q^{m(k-1)}-1)} - |\mP|>0.
\end{equation}
Finally, we count the elements of $\mP$ as
\begin{align*}
    2|\mP| &= \sum_{i=1}^m |\{(x,y) \in \mQ^2 \st x \neq y, \, \rk(y) =i, \, \sigma^\rk(x) \subseteq \sigma^\rk(y)\}| \\
    &= \sum_{i=1}^m \sum_{\substack{y \in \mQ \\ \rk(y)=i}} |\{x \in \mQ \st x \neq y, \,  \sigma^\rk(x) \subseteq \sigma^\rk(y)\}| \\
    &=\sum_{i=1}^m \frac{1}{q^m-1} \qbin{m}{i}{q} \, \prod_{j=0}^{i-1} (q^n-q^j) \left( \frac{q^{mi}-1}{q^m-1}-1  \right) \\
    &=\sum_{i=2}^m \frac{1}{q^m-1} \qbin{m}{i}{q} \, \prod_{j=0}^{i-1} (q^n-q^j) \left( \frac{q^{mi}-1}{q^m-1}-1  \right).
\end{align*}
Combining this with~\eqref{kk} concludes the proof.
\end{proof}

We now give a sufficient condition under which the assumption in Lemma~\ref{lem:exist} is satisfied. This gives us parameter ranges for which minimal codes exist. Note that, in line with what we observed in the Introduction of this paper, the next result does not depend on the field size~$q$. This behaviour of minimal rank-metric codes is in sharp contrast with analogous results for minimal codes in the Hamming metric; see e.g.~\cite[Theorem 2.14]{alfarano2020three}.

\begin{corollary}\label{cor:existenceres}
For every $m, k \ge 2$, there exists a minimal $[2k+m-2,k]_{q^m/q}$ code.
\end{corollary}

\begin{proof}
Fix an integer $n \ge k$ and observe that 
$$\frac{(q^{mn}-1)(q^{m(n-1)}-1)}{(q^{mk}-1)(q^{m(k-1)}-1)} \ge q^{mn+m(n-1)-mk-m(k-1)}=q^{2m(n-k)}.$$
Therefore the quantity in \eqref{eq:existence} can be bounded from below as follows:
\begin{align*}
   \eqref{eq:existence} & \geq   
    q^{2m(n-k)}- \frac{1}{2(q^m-1)^2} \, \sum_{i=2}^m  \qbin{m}{i}{q}\cdot q^{\binom{i}{2}} \cdot (q^{mi}-q^m)\, \prod_{j=0}^{i-1} (q^{n-j}-1)\\
    & >  q^{2m(n-k)}- \frac{1}{2(q^m-1)^2} \, \sum_{i=2}^m  \qbin{m}{i}{q}\cdot q^{\binom{i}{2}} \cdot q^{mi}\, \prod_{j=0}^{i-1} q^{n-j}\\
    & =  q^{2m(n-k)}- \frac{1}{2(q^m-1)^2} \, \sum_{i=2}^m  \qbin{m}{i}{q}\cdot q^{i(m+n)}=:t_q(m,n,k).
\end{align*}
Define the function
 $$ f(q):=\prod_{i=1}^{\infty}\frac{q^i}{q^i-1}.$$  In the sequel, we will use the following estimates:
\begin{align}
    \qbin{a}{b}{q}&<f(q)\, q^{b(a-b)}, && \hspace{-3em} \mbox{for } a,b \in \mathbb N, \label{est1}\\
%    \frac{q^a-1}{q^b-1}&\leq \frac{q}{q-1}q^{a-b}, && \hspace{-3em} \mbox{for } a,b \in \mathbb N, \ b < a, \label{est2} 
%    \\
   q^{e_1}+\ldots+q^{e_r}&\leq \frac{q}{q-1}q^{e_r}, &&\hspace{-3em} \mbox{for } e_i \in \mathbb Z, \ 0\leq e_1<\ldots<e_r. \label{est3}
\end{align}
We have
\begin{align}\label{eq:aux_ineq_tq}
    2(q^m-1)^2t_q(m,n,k)&\stackrel{\phantom{(6.6)}}{=}2(q^m-1)^2q^{2m(n-k)}-q^{m(m+n)}-\sum_{i=2}^{m-1}\qbin{m}{i}{q}q^{i(m+n)}  \nonumber \\
    &\stackrel{\eqref{est1}}{>}2(q^m-1)^2q^{2m(n-k)}-q^{m(m+n)}-f(q)\sum_{i=2}^{m-1}q^{i(2m+n-i)}  \nonumber \\
    &\stackrel{\eqref{est3}}{>}2(q^m-1)^2q^{2m(n-k)}-q^{m(m+n)}-\frac{qf(q)}{q-1}q^{(m-1)(m+n-1)}.  \nonumber \\
     &\stackrel{\phantom{(6.6)}}{>}2(q^m-1)^2q^{2m(n-k)}-q^{m(m+n)}-q^{(m-1)(m+n-1)+3},
\end{align}
where the last inequality follows from the fact that $f(q)<4$ for every prime power $q$.

We now specialize the argument to $n=2k+m-2$, proving that $t_q(m,2k+m-2,k)>0$ for every $m,k\geq 2$ and prime power $q$. 
% In the following chain of inequality, the one marked with~$*$ follows from \eqref{est1} and the one marked with~$**$ follows from \eqref{est3}. 
Using \eqref{eq:aux_ineq_tq} we find
\begin{align*}
    2(q^m-1)^2t_q(m,2k+m-2,k)& > 2(q^m-1)^2q^{2m(m+k-2)}-q^{2m(m+k-1)}-q^{(m-1)(2m+2k-3)+3} \\
    &= 2(q^m-1)^2q^{2m(m+k-2)}-(1+q^{-3m-2k+6})q^{2m(m+k-1)} \\ 
    & \geq 2(q^m-1)^2q^{2m(m+k-2)}-(1+q^{-4})q^{2m(m+k-1)} \\
    &= q^{2m(m+k-2)-4}\left(2(q^m-1)^2q^4-(q^4+1)q^{2m}\right)
\end{align*}
Hence  $t_q(m,2k+m-2,k)>0$  whenever 
$q^{2m+4}-4q^{m+4}-q^{2m}+2q^{4}\geq 0,$
which holds for every $m \geq 2$ and every prime power $q$.
Therefore there exists a minimal $[2k+m-2,k]_{q^m/q}$ code by Lemma~\ref{lem:exist}.
\end{proof}

\begin{remark}
Fix integers $k,m\geq 2$.
Then  Corollary \ref{coro:lowerboundlength} 
tells us that
for any length value $n<k+m-1$ an $\Fmk$ minimal code cannot exist, for any field size $q$. On the other hand, by Corollary \ref{cor:existenceres} for $n\geq 2k+m-2$ there exist $\Fmk$ minimal codes for every field size $q$. Therefore the existence of $\Fmk$  minimal codes remains in general an open question only for 
$k+m-1 \le n \le 2k+m-3$. 
%The only exceptions we are aware of are the $[m+2,3]_{q^m/q}$ codes constructed in Section \ref{sec:three_dim_minimal}. 
\end{remark}

\subsection{The Linearity Index of a $q$-System}

% \red{Da qui in poi ho generalizzato il lower bound tirando in mezzo i generalized rank weight. Cosi' possiamo motivare il fatto che li studiamo! }

Given an $\Fmk$ system $\mU$, one could be interested in understanding how $\mU$ is related to $\Fm$-subspaces of $\Fm^k$ and not only to $\Fm$-hyperplanes. This indeed could reveal some additional information on its parameters and whether it can be a  linear cutting blocking set or not. In this subsection we define and analyze a new parameter of projective system and with its aid we generalize the lower bound in Corollary~\ref{coro:lowerboundlength} for the length of minimal codes. 

\bigskip

Let $\mU$ be a $\Fmk$ system. We introduce a measure for the ``linearity'' of $\mU$ over $\F_{q^m}$. More precisely, we define the \textbf{linearity index} of $\mU$ as 
$$\ell(\mU)=\max\{ \dim_{\Fm}(H) \st H\subseteq \Fm^k \mbox{ is an $\F_{q^m}$-subspace}, \, H\subseteq \mU \}.$$

Observe that the value of $\ell(\mU)$  is invariant under equivalence of $\Fmk$ systems. In particular, it is a well-defined structural parameter of the corresponding equivalence class $[\mU]$. The following result relates $\ell(\mU)$ to the generalized rank weight of a code that gives rise to the $q$-system $\mU$.

\begin{lemma}\label{lem:pippo1-1}
Let $\mC$ be a nondegenerate $\Fmk$ code, and let $\mU$ be any corresponding $\Fmk$ system. Then 
$$ \ell(\mU)=k-\min\{r \st \drk_{r}(\mC)=n-(k-r)m\}.$$
\end{lemma}

\begin{proof}
 First of all, note that the set $\{r \st \drk_{r}(\mC)=n-(k-r)m\}$ is nonempty, since $\drk_k(\mC)=n$. By Theorem \ref{thm:genrankweight}, $$\drk_r(\mC)  =n- \max \left\{\dim_{\Fq}(\mU\cap H) \st H \mbox{ is an } \Fm\mbox{-subspace of codimension } r  \mbox{ in } \Fm^k \right\},$$
 which is equal to $n-(k-r)m$ if and only if there exists $H\subseteq \Fm^k$ of codimension $r$ contained in $\mU$.
\end{proof}

Lemma \ref{lem:pippo1-1} shows that the parameter $\ell$ is well-defined in the correspondence of Theorem~\ref{thm:1-1}. Hence, it is also a well-defined parameter of a nondegenerate  code $\mC$. Therefore, we will also refer to the \textbf{linearity index of a code $\mC$}, and denote it by $\ell(\mC)$.

\begin{lemma}\label{lem:pippo-genweight}
Let $\mC$ be a nondegenerate $\Fmk$ code with linearity index $\ell$.
 Then,  $\drk_{i+1}(\mC)-\drk_i(\mC)=m$ if and only if $i \geq k-\ell(\mC)$.
\end{lemma}

\begin{proof}
  \underline{$(\Leftarrow)$} This implication follows from the definition of $\ell(\mC)$.
  
  \underline{$(\Rightarrow)$} Let $\mU$ be any $\Fmk$ system associated to $\mC$. Let $H\subseteq \Fm^k$ be the space of codimension $i$ such that $\drk_i(\mC)=n-\dim_{\Fq}(H\cap \mU)$ and let $t:=\dim_{\Fq}(H\cap \mU)$.  Let $\mV:=H\cap \mU$ and observe that $|H^\prime \cap (\mV\setminus \{0\})| = (q^{t-m}-1)$ for any hyperplane $H^\prime$ in $H$. Let $\Lambda$ be the set of all hyperplanes in $H$, then by Lemma \ref{lem:st-eq}, we have
  
  $$\sum_{H^\prime\in\Lambda} |H^\prime \cap (\mV\setminus \{0\})| = \binom{k-i}{1}_{q^m}(q^{t-m}-1).$$
  
  Moreover, observe that every nonzero element of $\mV$ belongs to exactly $\binom{k-i-1}{1}_{q^m}$ hyperplanes in $H$. Hence,
  
  $$\sum_{H^\prime\in\Lambda} |H^\prime \cap (\mV\setminus \{0\})| = \sum_{v\in\mV\setminus\{0\}} |\{H^\prime \st v\in H^\prime\}| = \binom{k-i-1}{1}_{q^m}(q^t-1).$$
  
  By a double counting argument, we then have that 
  
  $$(q^{(k-i)m}-1)(q^{t-m}-1)=(q^t-1)(q^{(k-i-1)m}-1).$$
  By Lemma \ref{lem:tpowers}, it follows that $t=(k-i)m$. In particular, $\mV$ is an $[km,k]_{q^m/q}$ system associated to a simplex code. We conclude then that $H\subseteq \mV$ and then $H\subseteq \mU$, which implies that $\ell(\mC) \geq k-i$.
  
\end{proof}

\begin{proposition}\label{prop:pippo}
  Let $\mC$ be a nondegenerate $\Fmk$ code. Then 
  $$ \ell(\mC)\geq n-k(m-1).$$
\end{proposition}

\begin{proof}
Observe that $\sum_{i=0}^{k-1} \drk_{i+1}(\mC)-\drk_i(\mC) = \drk_k(\mC) - \drk_0(\mC) = n$. Moreover, by applying Lemma \ref{lem:pippo-genweight}, we have that
\begin{align*}
    \sum_{i=0}^{k-1} \drk_{i+1}(\mC)-\drk_i(\mC) &= \sum_{i=0}^{k-\ell(\mC)-1}\drk_{i+1}(\mC)-\drk_i(\mC) +  \sum_{i=k-\ell(\mC)}^{k-1}\drk_{i+1}(\mC)-\drk_i(\mC)\\ 
    &\leq (m-1)(\ell(\mC)-1) + m\ell(\mC)=m(k-1) + \ell(\mC)-m. \qedhere
\end{align*}
\end{proof}

The linearity index of a code can help in characterizing and finding improved bounds on the other parameters of a minimal $\Fmk$ code.

\begin{lemma}\label{lem:quotientcutting}
 Let $\mU$ be a  linear cutting $[n,k]_{q^m/q}$ blocking set. Suppose that there exists an $\ell$-dimensional $\Fm$-subspace $T$ of $\Fm^k$ such that $T\subseteq\mU$. Then $\mU/T$ is isomorphic to a linear cutting $[n-\ell m,k-\ell]_{q^m/q}$ blocking set. 
\end{lemma}

\begin{proof}
 By Proposition \ref{prop:cutting}, we need to show that for every $\Fm$-hyperplane $\bar{H}$ of $\Fm^k/T$ we have $\langle \bar{H}\cap\mU/T\rangle_{\Fm}=\bar{H}$. The $\Fm$-hyperplanes of $\Fm^k/T$ correspond to the $\Fm$-hyperplanes of $\Fm^k$ that contain $T$. Let $\bar{H}$ be an $\Fm$-hyperplane of $\Fm^k/T$. Then there exists an $\Fm$-hyperplane~$H$ of~$\Fm^k$ such that $\bar{H}=H/T$. Hence, 
 $$ \langle \bar{H}\cap\mU/T\rangle_{\Fm}=\langle (H\cap\mU)/T\rangle_{\Fm}=\langle H\cap\mU\rangle_{\Fm}/T=H/T=\bar{H},$$
 where the second last equality follows from the fact that $\mU$ is a linear cutting $[n,k]_{q^m/q}$ blocking set and by Proposition \ref{prop:cutting}.
\end{proof}

The following result is a generalization of 
Corollary~\ref{coro:lowerboundlength}.

\begin{proposition}\label{prop:genlowerbound}
 Let $\mU$ be a linear cutting $[n,k]_{q^m/q}$ blocking set and let $\ell$ be its linearity index. If $k-\ell\geq 2$, then
 $$ n-k\geq (\ell+1)(m-1).$$
 In particular, for every $1\leq r \leq k-\lfloor\frac{n-k+1}{m-1}\rfloor-1$, we have $\drk_{r}(\mC)>n-rm$, where $\mC$ is the nondegnerate $\Fmk$ code associated to $\mU$.
\end{proposition}

\begin{proof}
 Let $T\subseteq \Fm^k$ be an $\ell$-dimensional $\Fm$-subspace contained in $\mU$. By Lemma \ref{lem:quotientcutting} we have that $\mU/T$ is isomorphic to a linear cutting $[n-\ell m,k-\ell]_{q^m/q}$ blocking set. Therefore, by Corollary~\ref{coro:lowerboundlength}, we obtain 
 $$ n-\ell m\geq k-\ell+m-1,$$
 from which we derive the desired inequality.
 
 \noindent
 For the second part, if $\ell$ does not satisfy the above inequality, i.e. if $\ell\geq \lfloor\frac{n-k+1}{m-1}\rfloor+1$, then~$\mU$ cannot contain any $\ell$-dimensional $\Fm$-subspace. This is equivalent to say that $\drk_{k-\ell}(\mC)>n-(k-\ell)m$.
\end{proof}

\begin{remark}
 As a consequence of \ref{prop:genlowerbound}, it can be immediately seen that in order to construct short minimal rank-metric code, one has to try to construct linear cutting blocking sets not containing $\Fm$-subspaces. This is also consistent with the construction of minimal $[m+2,3]_{q^m/q}$ codes provided in Section \ref{sec:three_dim_minimal}. Indeed, if a $\Fmk$ system $\mU$ contains a $\Fm$-subspace $H$, then in the associated linear set $L_{\mU}$ one has $\wt_{\mU}(P)=m$ for every $P\in\PG(H,\Fm)$. In particular, the associated linear set is far from being scattered.  
\end{remark}

Proposition \ref{prop:genlowerbound} allows to characterize nondegenerate $[(k-1)m,k]_{q^m/q}$ minimal codes.

\begin{corollary}\label{cor:characterization.k-1m}
 Let $k\geq 2$ and $\mC$ be a nondegenerate $[(k-1)m,k]_{q^m/q}$ code with linearity index $\ell=\ell(\mC)$. The following are equivalent.
 \begin{enumerate}
     \item $\mC$ is minimal.
     \item $\ell<k-2$.
     \item $\drk_2(\mC)>m$.
 \end{enumerate} 
\end{corollary}

\begin{proof}
 \underline{$(1)\Rightarrow (2)$}: First observe that $\ell$ can not be equal to $k-1$. Indeed, if $\ell =k-1$, then the code $\mC$ is not $k$-dimensional. Hence, $k-\ell \geq 2$ and since $\mC$ is minimal, then by Proposition \ref{prop:genlowerbound} it holds $$(k-1)m-k+1 > (\ell+1)(m-1),$$
 from which we deduce $\ell<k-2$.
 
 \underline{$(2)\Rightarrow (1)$}: Suppose $\mC$ is not minimal and let $\mU$ be any associated $[(k-1)m,k]_{q^m/q}$ system. By Proposition \ref{prop:cutting}, there exists an $\Fm$-hyperplane of $\Fm^k$ such that $\langle H\cap \mU\rangle_{\Fm}=:H'$, with $\dim_{\Fm}(H')\leq k-2$. Hence, we obtain
 \begin{align*}
     (k-2)m & \geq \dim_{\Fq}(H')\geq \dim_{\Fq}(H\cap \mU) \\
     &= \dim_{\Fq}(H)+\dim_{\Fq}(\mU)-\dim_{\Fq}(\mU +H) \\
     & \geq (k-1)m+(k-1)m-km=(k-2)m. 
 \end{align*}
 Hence, all the inequalities above are equalities and $H'=H\cap \mU$. This implies that $\mU$ contains the $(k-2)$-dimensional $\Fm$-subspace $H'$ and $\ell \geq k-2$. 
 
 \underline{$(2)\Leftrightarrow (3)$}: Observe that $\ell \leq k-2$ and $\drk_2(\mC) \geq m$. Then, the statement directly follows from Lemma \ref{lem:pippo1-1}.
\end{proof}

\begin{corollary}
Let $k\geq 3$ be an integer. A nondegenerate $[(k-1)m,k]_{q^m/q}$ minimal code exists if and only if $m\geq 3$.
\end{corollary}

\begin{proof}
 Suppose that $m\geq 3$ and construct the $[(k-3)m+3(m-1),k]_{q^m/q}$ system~$\mU^\prime$ as follows. Take $\mV^\prime=\langle \alpha^ie_j \st 0\leq i \leq m-2, k-2\leq j\leq k\rangle$, where $\alpha \in \Fm$ is such that $\Fq(\alpha)=\Fm$. Then, consider 
 $\mU^\prime=\{(\,v \, \mid \, 0,0,0) \st v \in \Fm^{k-3}\} \oplus \mV^\prime$.
 By construction $\ell(\mU^\prime)=k-3$. Moreover, since $m\geq 3$ then $(k-3)m+3(m-1)\geq (k-1)m$, and we can take any $(k-1)m$-dimensional $\Fq$-subspace $\mU$ of $\mU^\prime$, which has $\ell(\mU)\leq k-3$ and by Corollary \ref{cor:characterization.k-1m} is minimal.
 
 Assume now $m\leq 2$, and let $\mC$ be a nondegenerate  a $[(k-1)m,k]_{q^m/q}$ code. Then by Proposition \ref{prop:pippo}, we have $\ell(\mC)\geq k-m\geq k-2$. Hence, By Corollary \ref{cor:characterization.k-1m}, $\mC$ is not minimal.
\end{proof}

\section*{Acknowledgements}
The authors of this paper would like to thank John Sheekey and Ferdinando Zullo for fruitful discussions and comments. 

\bigskip

\bibliographystyle{abbrv}
\bibliography{references.bib}

\begin{thebibliography}{10}

\bibitem{MR1768542}
R.~Ahlswede, N.~Cai, S.-Y.~R. Li, and R.~W. Yeung.
\newblock Network information flow.
\newblock {\em IEEE Trans. Inform. Theory}, 46(4):1204--1216, 2000.

\bibitem{alfarano2019geometric}
G.~N. Alfarano, M.~Borello, and A.~Neri.
\newblock A geometric characterization of minimal codes and their asymptotic
  performance.
\newblock {\em Adv. Math. Commun.}, 2020.

\bibitem{alfarano2020three}
G.~N. Alfarano, M.~Borello, A.~Neri, and A.~Ravagnani.
\newblock Three combinatorial perspectives on minimal codes.
\newblock {\em arXiv preprint arXiv:2010.16339}, 2020.

\bibitem{ashikhmin1998minimal}
A.~Ashikhmin and A.~Barg.
\newblock Minimal vectors in linear codes.
\newblock {\em IEEE Trans. Inform. Theory}, 44(5):2010--2017, 1998.

\bibitem{BALL2000294}
S.~Ball, A.~Blokhuis, and M.~Lavrauw.
\newblock Linear {$(q+1)$}-fold blocking sets in {$\mathrm{PG}(2, q^4)$}.
\newblock {\em Finite Fields Appl.}, 6(4):294--301, 2000.

\bibitem{bartoli2020evasive}
D.~Bartoli, B.~Csajb{\'o}k, G.~Marino, and R.~Trombetti.
\newblock Evasive subspaces.
\newblock {\em arXiv preprint arXiv:2005.08401}, 2020.

\bibitem{bartoli2018maximum}
D.~Bartoli, M.~Giulietti, G.~Marino, and O.~Polverino.
\newblock Maximum scattered linear sets and complete caps in {G}alois spaces.
\newblock {\em Combinatorica}, 38(2):255--278, 2018.

\bibitem{berger2003isometries}
T.~P. Berger.
\newblock Isometries for rank distance and permutation group of {G}abidulin
  codes.
\newblock {\em IEEE Trans. Inform. Theory}, 49(11):3016--3019, 2003.

\bibitem{blokhuis2000scattered}
A.~Blokhuis and M.~Lavrauw.
\newblock Scattered spaces with respect to a spread in {${\rm PG}(n,q)$}.
\newblock {\em Geom. Dedicata}, 81(1-3):231--243, 2000.

\bibitem{bonini2020minimal}
M.~Bonini and M.~Borello.
\newblock Minimal linear codes arising from blocking sets.
\newblock {\em J. Algebraic Combin.}, 53(2):327--341, 2021.

\bibitem{bonisoli1983}
A.~Bonisoli.
\newblock Every equidistant linear code is a sequence of dual {H}amming codes.
\newblock {\em Ars Combin.}, 18:181--186, 1984.

\bibitem{brouwer1982blocking}
A.~E. Brouwer and H.~A. Wilbrink.
\newblock Blocking sets in translation planes.
\newblock {\em J. Geom.}, 19(2):200, 1982.

\bibitem{MR939470}
A.~A. Bruen, J.~A. Thas, and A.~Blokhuis.
\newblock On {M}.{D}.{S}. codes, arcs in {${\rm PG}(n,q)$} with {$q$} even, and
  a solution of three fundamental problems of {B}. {S}egre.
\newblock {\em Invent. Math.}, 92(3):441--459, 1988.

\bibitem{MR818812}
R.~Calderbank and W.~M. Kantor.
\newblock The geometry of two-weight codes.
\newblock {\em Bull. London Math. Soc.}, 18(2):97--122, 1986.

\bibitem{csajbok2017maximum}
B.~Csajb\'{o}k, G.~Marino, O.~Polverino, and F.~Zullo.
\newblock Maximum scattered linear sets and {MRD}-codes.
\newblock {\em J. Algebraic Combin.}, 46(3-4):517--531, 2017.

\bibitem{1930-5346_2011_1_119}
A.~A. Davydov, M.~Giulietti, S.~Marcugini, and F.~Pambianco.
\newblock Linear nonbinary covering codes and saturating sets in projective
  spaces.
\newblock {\em Adv. Math. Commun.}, 5(1):119--147, 2011.

\bibitem{de2018weight}
J.~de~la Cruz, E.~Gorla, H.~H. L\'{o}pez, and A.~Ravagnani.
\newblock Weight distribution of rank-metric codes.
\newblock {\em Des. Codes Cryptogr.}, 86(1):1--16, 2018.

\bibitem{delsarte1978bilinear}
P.~Delsarte.
\newblock Bilinear forms over a finite field, with applications to coding
  theory.
\newblock {\em J. Combin. Theory Ser. A}, 25(3):226--241, 1978.

\bibitem{fancsali2014lines}
S.~Fancsali and P.~Sziklai.
\newblock Lines in higgledy-piggledy arrangement.
\newblock {\em Electron. J. Comb.}, 21(2), 2014.

\bibitem{gabidulin1985theory}
E.~M. Gabidulin.
\newblock Theory of codes with maximum rank distance.
\newblock {\em Probl. Peredachi Informatsii}, 21(1):3--16, 1985.

\bibitem{MR1345289}
A.~Garc\'{\i}a and H.~Stichtenoth.
\newblock A tower of {A}rtin-{S}chreier extensions of function fields attaining
  the {D}rinfeld-{V}ladut bound.
\newblock {\em Invent. Math.}, 121(1):211--222, 1995.

\bibitem{giorgetti2010galois}
M.~Giorgetti and A.~Previtali.
\newblock Galois invariance, trace codes and subfield subcodes.
\newblock {\em Finite Fields Appl.}, 16(2):96--99, 2010.

\bibitem{gruica2020common}
A.~Gruica and A.~Ravagnani.
\newblock Common complements of linear subspaces and the sparseness of {MRD}
  codes.
\newblock {\em arXiv:2011.02993}, 2020.

\bibitem{MR3857591}
Z.~Heng, C.~Ding, and Z.~Zhou.
\newblock Minimal linear codes over finite fields.
\newblock {\em Finite Fields Appl.}, 54:176--196, 2018.

\bibitem{huffman2010fundamentals}
W.~C. Huffman and V.~Pless.
\newblock {\em Fundamentals of error-correcting codes}.
\newblock Cambridge University Press, Cambridge, 2003.

\bibitem{jurrius2017defining}
R.~Jurrius and R.~Pellikaan.
\newblock On defining generalized rank weights.
\newblock {\em Adv. Math. Commun.}, 11(1):225--235, 2017.

\bibitem{MR2451015}
R.~K\"{o}tter and F.~R. Kschischang.
\newblock Coding for errors and erasures in random network coding.
\newblock {\em IEEE Trans. Inform. Theory}, 54(8):3579--3591, 2008.

\bibitem{MR1966785}
S.-Y.~R. Li, R.~W. Yeung, and N.~Cai.
\newblock Linear network coding.
\newblock {\em IEEE Trans. Inform. Theory}, 49(2):371--381, 2003.

\bibitem{lunardon1999normal}
G.~Lunardon.
\newblock Normal spreads.
\newblock {\em Geom. Dedicata}, 75(3):245--261, 1999.

\bibitem{macwilliams1977theory}
F.~J. MacWilliams and N.~J.~A. Sloane.
\newblock {\em The theory of error correcting codes}, volume~16.
\newblock Elsevier, 1977.

\bibitem{martinez2016similarities}
U.~Mart\'{\i}nez-Pe\~{n}as.
\newblock On the similarities between generalized rank and {H}amming weights
  and their applications to network coding.
\newblock {\em IEEE Trans. Inform. Theory}, 62(7):4081--4095, 2016.

\bibitem{martinez2020hamming}
U.~Mart\'{\i}nez-Pe\~{n}as.
\newblock Hamming and simplex codes for the sum-rank metric.
\newblock {\em Des. Codes Cryptogr.}, 88(8):1521--1539, 2020.

\bibitem{Massey}
J.~L. Massey.
\newblock Minimal codewords and secret sharing.
\newblock In {\em Proceedings of the 6th joint Swedish-Russian international
  workshop on information theory}, pages 276--279, 1993.

\bibitem{polverino2010linear}
O.~Polverino.
\newblock Linear sets in finite projective spaces.
\newblock {\em Discrete Math.}, 310(22):3096--3107, 2010.

\bibitem{polverino2020connections}
O.~Polverino and F.~Zullo.
\newblock Connections between scattered linear sets and {MRD}-codes.
\newblock {\em Bulletin of the ICA}, 89:46--74, 2020.

\bibitem{randrianarisoa2020geometric}
T.~H. Randrianarisoa.
\newblock A geometric approach to rank metric codes and a classification of
  constant weight codes.
\newblock {\em Des. Codes Cryptogr.}, 88(7):1331--1348, 2020.

\bibitem{ravagnani2016generalized}
A.~Ravagnani.
\newblock Generalized weights: an anticode approach.
\newblock {\em J. Pure Appl. Algebra}, 220(5):1946--1962, 2016.

\bibitem{ravagnani2016rank}
A.~Ravagnani.
\newblock Rank-metric codes and their duality theory.
\newblock {\em Des. Codes Cryptogr.}, 80(1):197--216, 2016.

\bibitem{sheekey2016new}
J.~Sheekey.
\newblock A new family of linear maximum rank distance codes.
\newblock {\em Adv. Math. Commun.}, 10(3):475--488, 2016.

\bibitem{sheekey2019scatterd}
J.~Sheekey.
\newblock ({S}cattered) {L}inear {S}ets are to {R}ank-{M}etric {C}odes as
  {A}rcs are to {H}amming-{M}etric {C}odes.
\newblock In M.~Greferath, C.~Hollanti, and J.~Rosenthal, editors, {\em
  Oberwolfach Report No. 13/2019}, 2019.

\bibitem{Sheekey2020}
J.~Sheekey and G.~Van~de Voorde.
\newblock Rank-metric codes, linear sets, and their duality.
\newblock {\em Des. Codes Cryptogr.}, 88(4):655--675, 2020.

\bibitem{MR2450762}
D.~Silva, F.~R. Kschischang, and R.~K\"{o}tter.
\newblock A rank-metric approach to error control in random network coding.
\newblock {\em IEEE Trans. Inform. Theory}, 54(9):3951--3967, 2008.

\bibitem{tang}
C.~Tang, Y.~Qiu, Q.~Liao, and Z.~Zhou.
\newblock Full characterization of minimal linear codes as cutting blocking
  sets.
\newblock {\em IEEE Trans. Inform. Theory}, 67(6):3690--3700, 2021.

\bibitem{tsfasman1991algebraic}
M.~A. Tsfasman and S.~G. Vl{\u{a}}du{\c{t}}.
\newblock {\em Algebraic-geometric codes}, volume~58 of {\em Mathematics and
  its Applications (Soviet Series)}.
\newblock Kluwer Academic Publishers Group, Dordrecht, 1991.
\newblock Translated from the Russian by the authors.

\bibitem{zini2021scattered}
G.~Zini and F.~Zullo.
\newblock Scattered subspaces and related codes.
\newblock {\em Des. Codes Cryptogr.}, pages 1--21, 2021.

\end{thebibliography}

\bigskip

\section*{Appendix}\label{sec:appendix}

\begin{proof}[Proof of Theorem \ref{thm:1-1}] We prove a series of properties separately.
\begin{itemize}
    \item $\Phi([\mC])$ does not depend on the choice of the generator matrix $G$.
    Indeed, if $G^\prime$ is another generator matrix for $\mC$ then there is an $\Fm$-linear map $\varphi$, such that $\varphi(G)=G^\prime$. The same map sends the $\Fq$-columnspace of $G$ into the $\Fq$-columnspace of $G^\prime$.
    \item $\Phi([\mC])$ does not depend on $\mC$ but only on its equivalence class. To see this, let $\mC^\prime$ be a code linearly equivalent to $\mC$, then there is a matrix $A\in \GL_n(q)$ such that $\mC^\prime = \mC\cdot A$. Hence, if $G$ is a generator matrix for $\mC$, then $GA$ is a generator matrix for $\mC^\prime$ and they have the same $\Fq$-columnspace. Hence, the map $\Phi$ does not depend on the choice of the representative.
   
    \item $\Phi([\mC]) \in \mU\Fmkd$. To see this, let 
    $[n',k',d']$ be the parameters of $\Phi([\mC])$. We need to show that $(n,k,d)=(n',k',d')$. Since $\mC$ has dimension $k$ over $\F_{q^m}$ we have $k=k'$.

    In order to prove that $n=n'$, we use the fact that $\mC$ is nondegenerate by assumption. More precisely, let $G$ be a generator matrix for $\mC$. Since $\mC$ is nondegenerate, by Proposition~\ref{prop:newnondegeneracy}, $n=\dim(\sigma^\rk(\mC))$ is equal to the dimension of the $\Fq$-space of the columns of $G$, that is $n^\prime$.
    %for any $A\in\GL_n(q)$, $\mC\cdot A$ is Hamming-nondegenerate. Hence, the columns of $G$ are linearly independent over $\Fq$. This implies that 
    
    Finally, denote by $G$ a generator matrix of $\mC$. By Lemma~\ref{lem:rkvG}, for all nonzero $v \in \Fm^k$ we have
    $$
        \rk(vG) = n - \dim_{\F_q}(\Phi([\mC]) \cap \langle v \rangle^\perp).$$
    As $v$ ranges over the nonzero vectors in $\Fm^k$, $\langle v \rangle^\perp$ ranges over all
    $\Fm$-hyperplanes in $\Fm^k$. Therefore $d=d'$
    by definition of $d'$.

    \item $\Psi([\mU])$ does not depend on $\mU$ but only on its equivalence class. To see this, assume $\mU^\prime$ is an $[n,k]_{q^m/q}$ system equivalent to $\mU$, hence, there is an $\Fm$-isomorphism $\phi$, such that $\phi(\mU)= \mU^\prime$. In particular, if $\{g_1,\dots,g_n\}$ is a basis of $\mU$ and $\{g_1^\prime,\dots,g_n^\prime\}$ is a basis of $\mU^\prime$, then $\phi(\{g_1^\prime,\dots,g_n^\prime\}) =\{g_1^\prime,\dots,g_n^\prime\}$. In particular, let $G$ be the matrix whose $i$-th column is given by $g_i$ and $G^\prime$ be the matrix whose $i$-th column is given by $g_i^\prime$, then there is a matrix $A\in \GL_n(q)$, such that $G^\prime=GA$. Hence, the rank-metric codes generated by $G$ and $G^\prime$ are linearly equivalent. So, we conclude that $\Psi$ does not depend on the choice of~$\mU$.
    
    \item $\Psi([\mU]) \in \mC\Fmkd$. To see this, let $\{g_1,\dots, g_n\}$ be an $\Fq$-basis of $\mU$ and let $\mC$ be the $[n^\prime, k^\prime, d^\prime]$ code whose generator matrix $G$ has $g_i$ as $i$-th column. Then, obviously, the length of $\mC$ is $n^\prime=n$. The rows of $G$ are linearly independent over $\Fm$ otherwise there is $x\in\Fm^k$ such that $xg_i^\top=0$ for all $i$. Hence, $x$ defines an hyperplane containing $\mU$, which contradicts the fact that $\langle \mU \rangle_{\Fm}=\Fm^k$. This ensures that the dimension $k^\prime$ of $\mC$ is equal to $k$.
    For a matrix $G\in\Fm^{k\times n}$ we denote by $G_i$ the $i$-th column of~$G$. Now, for the distance, observe that  for all $v\in\Fm^k$,
    \begin{align*}
        d^\prime &= \min\{\wH(vGA) \st A\in \GL_n(q)\} \\
        &= \min\{n-|\{i \st (GA)_i \in \langle v \rangle^\perp\}|\} \\
        &= n- \max \{\dim_{\Fq}(\mU\cap H) \st H \mbox{ is an } \Fm\mbox{-hyperplane in }  \Fm^k \},
    \end{align*}
    where the last equality follows from Equation~\eqref{eq:prel}.
    Finally, since the $g_i$'s are linearly independent over $\Fq$,
    $\mC$ is nondegenerate by Proposition~\ref{prop:newnondegeneracy}.

\end{itemize}

All of this establishes the desired result.
 \end{proof}

\subsection*{Generalized rank weights}
In order to prove Theorem \ref{thm:genrankweight}, we first recall the notion of generalized Hamming weight. Given an~$\Fmk$ nondegenerate code~$\mC$, for every $r=1,\ldots, k$, the  \textbf{$r$-th generalized Hamming weight} of $\mC$ is defined as 
$$ d^\HH_r(\mC) = \min\{|\sH(V)| \st V\subseteq \mC, \dim(V)= r\}.$$

It is easy to see that
for an $\Fmk$ rank-metric code $\mC$ one has
\begin{equation}\label{eq:genrankandhamm}
 \drk_r(\mC)=\min\{\dH_r(\mC\cdot A) \st A\in\GL_n(q)\};
 \end{equation}
 see e.g.~\cite[Theorem 2]{martinez2016similarities}. Recall also the following well-known result.
 
  \begin{lem}[see \textnormal{\cite[Theorem 1.1.14]{tsfasman1991algebraic}}]\label{prop:drGintersectH}
Let $\mC$ be an $[n,k]_{q^m/q}$ code and $G$ be a generator matrix for~$\mC$. Then
$$  \dH_r(\mC) = \min\{n-| \{i \st G_i\in H\}| \st H\leq\Fm^k, \ \dim H \leq k-r\},$$
where $G_i$ denotes the $i$-th column of $G$.
\end{lem}

\begin{proof}[Proof of Theorem \ref{thm:genrankweight}]
Let $G$ be a generator matrix of $\mC$. Then, by the previous Lemma and Equation \eqref{eq:genrankandhamm} we obtain that $$\drk_r(\mC)=n-\max \{| \{i \st (GA)_i\in H\}| \st A\in\GL_n(q), \, H\leq\Fm^k, \, \dim H\leq k-r\}.$$ 

Let $\mU$ be the $\Fq$-span of the columns of $G$, i.e. $\mU$ is an $[n,k]_{q^m/q}$ system corresponding to the equivalence class of $\mC$. Note that, by Equation~\eqref{eq:prel}, for a fixed $H\subseteq\Fm^k$ with $\dim H \leq k-r$
we have that 
$$\max\{| \{i \st (GA)_i\in H\}| \st A\in\GL_n(q)\} = \dim_{\Fq}(\mU \ \cap H).$$
This concludes the proof. 
\end{proof}

\end{document}